\documentclass[10pt]{amsart}

\usepackage{amsmath}
\usepackage{amssymb}
\usepackage{amsthm}
\usepackage{bm}
\usepackage{graphicx}
\usepackage{psfrag}
\usepackage{color}
\usepackage{hyperref}
\hypersetup{colorlinks=true, linkcolor=blue, citecolor=magenta, urlcolor=wine}
\usepackage{url}
\usepackage{algpseudocode}
\usepackage{fancyhdr}
\usepackage{mathtools}
\usepackage{tikz-cd}
\usepackage{xy}
\input xy
\xyoption{all}
\usepackage{stmaryrd}
\usepackage{calrsfs}

\usepackage{mathrsfs}

\newtheorem*{maintheorem*}{Main Theorem}
\newtheorem{theorem}{Theorem}[section]
\newtheorem*{theorem*}{Main Theorem}

\newtheorem{question}[theorem]{Question}
\newtheorem{prop}[theorem]{Proposition}

\newtheorem{lemma}[theorem]{Lemma}
\newtheorem{cor}[theorem]{Corollary}

\theoremstyle{definition}

\newtheorem{remark}[theorem]{Remark}
\newtheorem{example}[theorem]{Example}
\numberwithin{equation}{section}

\newcommand{\pp}{\mathbb{P}}
\newcommand{\ppp}{\mathcal{P}}

\newcommand{\qq}{\mathbb{Q}}

\newcommand{\uu}{\mathcal{U}}
\newcommand{\zz}{\mathbb{Z}}

\newcommand{\rr}{\mathbb{R}}
\newcommand{\nn}{\mathbb{N}}
\newcommand{\cc}{\mathbb{C}}
\newcommand{\ldb}{\llbracket}
\newcommand{\rdb}{\rrbracket}

\newcommand{\ZZ}{\mathbb{Z}}

\newcommand{\NN}{\mathbb{N}}

\newcommand{\br}[1]{{\left\{ #1 \right\}}}
\newcommand{\pr}[1]{{\left( #1\right)}}

\newcommand{\Pfin}{\mathcal{P}_{\mathrm{fin}}}

\newcommand{\gp}{\mathrm{gp}}


\begin{document}

\mbox{}
\title{On finitary power monoids \\ of linearly orderable monoids}

\author{Jiya Dani}
\address{Liberal Arts and Science Academy\\Austin, TX  78721}
\email{cs.program2004@gmail.com}

\author{Felix Gotti}
\address{Department of Mathematics\\MIT\\Cambridge, MA 02139}
\email{fgotti@mit.edu}

\author{Leo Hong}
\address{Department of Mathematics\\UNCC\\Charlotte, NC 28223}
\email{lhong6@charlotte.edu}

\author{Bangzheng Li}
\address{Department of Mathematics\\MIT\\Cambridge, MA 02139}
\email{liben@mit.edu}

\author{Shimon Schlessinger}
\address{Harvard-Westlake School\\Los Angeles, CA 91604}
\email{shimonschlessinger@gmail.com}

%\date{\today}
	
\begin{abstract}
	A commutative monoid $M$ is called a linearly orderable monoid if there exists a total order on $M$ that is compatible with the monoid operation. The finitary power monoid of a commutative monoid $M$ is the monoid consisting of all nonempty finite subsets of $M$ under the so-called sumset. In this paper, we investigate whether certain atomic and divisibility properties ascend from linearly orderable monoids to their corresponding finitary power monoids.
\end{abstract}

\bigskip
\maketitle

%%%%%%%%%%%%%%%%%%%%%%
%%%%%%%%%%%%%%%%%%%%%%
\section{Introduction}
\label{sec:intro}

Let $M$ be an additively-written commutative monoid (i.e., a commutative semigroup with an identity element). The power monoid $\mathcal{P}(M)$ of $M$ is the commutative monoid consisting of all nonempty subsets of $M$ under the so-called sumset: for any nonempty subsets $S$ and $T$ of $M$,
\[
	S+T := \{s+t : s \in S \text{ and } t \in T \}.
\]
The finitary power monoid of $M$, denoted here by $\mathcal{P}_{\text{fin}}(M)$, is the submonoid of $\ppp(M)$ consisting of all finite nonempty subsets of $M$. Both power monoids and finitary power monoids have been investigated in the literature of semigroup theory for several decades. Refer to~\cite{tT86} and its references for results up to the 1980s, and to ~\cite{TY23} and its references for more recent results. In the scope of this paper, the algebraic objects we are interested in are finitary power monoids of linearly orderable monoids (i.e., commutative monoids that can be endowed with a total order compatible with their corresponding operations).
\smallskip

In the setting of power monoids, one problem that has received a great deal of attention is the isomorphism problem: this is the problem of deciding, given a class $\mathcal{C}$ of commutative monoids,  whether non-isomorphic monoids in $\mathcal{C}$ induce non-isomorphic (finitary) power monoids. A compendium of progress on the isomorphism problem and further problems in the context of power monoids can be found in~\cite{jP86} as well as in the works cited therein. The study of the isomorphism problem is still quite active. For instance, the isomorphism problem for power monoids of rank-$1$ torsion-free commutative monoids was recently solved in~\cite{sT24} (it follows from~\cite[Section~24]{lF70} and \cite[Theorem~2.9]{rG84} that every rank-$1$ torsion-free commutative monoid that is not a group can be realized as a Puiseux monoid (i.e., an additive submonoid of $\qq_{\ge 0}$)).
\smallskip

Arithmetic and factorization aspects of finitary power monoids were previously studied in~\cite{FT18}, while atomic and ideal-theoretical aspects of finitary power monoids were previously studied in~\cite{BG23} in the setting of numerical monoids (i.e., Puiseux monoids consisting of nonnegative integers). Another classical problem in the setting of (finitary) power monoids that has attracted the attention of several semigroup theorists for many years is that of the potential ascent of monoidal properties, which boils down to the following question: does the fact that a commutative monoid~$M$ satisfies a given property~$\mathfrak{p}$ imply that the (finitary) power monoid of~$M$ also satisfies the property~$\mathfrak{p}$? As for the isomorphism problem, progress in this direction until the eighties can be found in ~\cite{tT86} and in the papers it references. Moreover, in the recent paper~\cite{GLRRT24}, the authors investigate the ascent of atomic and factorization properties from Puiseux monoids to their corresponding finitary power monoids.
\smallskip

In this paper we investigate the ascent of ideal-theoretical and atomic properties from linearly orderable commutative monoids to their corresponding finitary power monoids. In Section~\ref{sec:prelim}, we briefly revise some fundamental notation and terminology we will use throughout this paper.
\smallskip

In Section~\ref{sec:AC of PI}, we study ascending chains of principal ideals in the setting of finitary power monoids. A commutative monoid $M$ satisfies the ascending chain condition on principal ideals (ACCP) if every ascending chain consisting of principal ideals of $M$ eventually stabilizes. Two conditions weaker than the ACCP were introduced by Li and the second author in~\cite{GL23}: the quasi-ACCP and the almost ACCP. The monoid $M$ satisfies the quasi-ACCP (resp., almost ACCP) if for every nonempty finite subset $S$ of $M$, there exists a common divisor (resp., an atomic common divisor) $d \in M$ of $S$ in $M$ and $s \in S$ such that every ascending chain of principal ideals starting at $s-d + M$ eventually stabilizes. We prove that these two properties ascend from any linearly orderable commutative monoid to its corresponding finitary power monoid.
\smallskip

In Section~\ref{sec:atomicity and MCD}, we study the connection between atomicity and the existence of maximal common divisors (MCD) in finitary power monoids of linearly orderable monoids. In~\cite{GLRRT24}, Gonzalez et al. proved that the property of being atomic does not ascend to finitary power monoids on the class of linearly orderable monoids: indeed, they constructed an atomic Puiseux monoid whose power monoid is not atomic. We provide an alternative construction to argue that atomicity does not ascend to finitary power monoids on the class of linearly orderable torsion-free commutative monoids. In the same section, we characterize the linearly orderable monoids whose corresponding finitary power monoids are atomic as follows: for a linearly orderable monoid $M$, the power monoid $\Pfin(M)$ is atomic if and only if $M$ is atomic and every nonempty finite subset of $M$ has an MCD.
\smallskip

In Section~\ref{sec:notions weaker than atomicity}, we investigate properties that are more general than atomicity. Following Lebowitz-Lockard~\cite{nLL19}, we say that the monoid $M$ is nearly atomic if there exists an $s \in M$ such that every element of $s+M$ is atomic. Every atomic monoid is nearly atomic. We produce an atomic Puiseux monoid whose finitary power monoid is not even nearly atomic and, as a result, we obtain that near atomicity does not ascend to finitary power monoids on the class of linearly orderable monoids. Quasi-atomicity and almost atomicity are two generalized notions of atomicity Boynton and Coykendall introduced in their study of graphs of divisibility~\cite{BC15}. Every nearly atomic monoid is almost atomic, while every almost atomic monoid is quasi-atomic. We settle the ascent of both quasi-atomicity and almost atomicity, constructing a rank-$2$ linearly orderable commutative monoid that is almost atomic but its finitary power monoid is not even quasi-atomic.
\smallskip

In Section~\ref{sec:Furstenberg and IDF}, which is the final section of this paper, we study the Furstenberg property. Following \cite{pC17}, we say that the monoid $M$ is a Furstenberg monoid if every non-invertible element of $M$ is divisible by an atom. The Furstenberg property, along with some generalizations, was recently studied in~\cite{nLL19} in the setting of integral domains. We first argue that the Furstenberg property, as well as some generalizations introduced in~\cite{nLL19}, ascend to finitary power monoids on the class of linearly orderable monoids. Following \cite{GW75}, we say that $M$ is an IDF-monoid if every non-invertible element of $M$ is divisible at most by finitely many atoms up to associate. An IDF-monoid that is atomic (resp., Furstenberg) is called a finite factorization monoid (FFM) (resp., a TIDF-monoid). The finite factorization property and the TIDF property were coined and first studied by Anderson, Anderson, and Zafrullah~\cite{AAZ90} and by Gotti and Zafrullah~\cite{GZ23}, respectively. We conclude this paper proving that both properties ascend to finitary power monoids on the class of Archimedean positive monoids.

\bigskip
%%%%%%%%%%%
%%%%%%%%%%%
\section{Preliminary}
\label{sec:prelim}

\medskip
%%%%%%%%%%%%%%%%%%%%%%%%%%%%%
\subsection{General Notation}

As is customary, $\zz$, $\qq$, $\rr$, and $\cc$ will denote the set of integers, rational numbers, real numbers, and complex numbers, respectively. We let $\nn$ and $\nn_0$ denote the set of positive and nonnegative integers, respectively. Also, we let $\pp$ denote the set of primes. For $b,c \in \zz$ with $b \le c$, we let $\ldb b,c \rdb$ denote the set of integers between~$b$ and $c$:
\[
	\ldb b,c \rdb = \{n \in \zz : b \le n \le c\}.
\]
In addition, for $S \subseteq \rr$ and $r \in \rr$, we set
\[
	S_{\ge r} := \{s \in S : s \ge r\} \quad \emph{ and } \quad S_{> r} := \{s \in S : s > r\}.
\]
For a nonzero $q \in \qq$, let $(n,d)$ be the unique pair with $n \in \zz$ and $d \in \nn$ such that $q = \frac nd$ and $\gcd(n,d) = 1$. We will denote $n$ and $d$ by $\mathsf{n}(q)$ and $\mathsf{d}(q)$, respectively, setting $\mathsf{d}(S) := \{\mathsf{d}(s) : s \in S\}$ for any subset $S$ of $\qq \setminus \{0\}$. For each $p \in \pp$ and $n \in \zz \setminus \{0\}$, we let $v_p(n)$ denote the $p$-adic valuation of $n$, that is, the maximum $m \in \nn_0$ such that $p^m \mid n$, and for $q \in \qq \setminus \{0\}$, we set $v_p(q) := v_p(\mathsf{n}(q))-v_p(\mathsf{d}(q))$ (after defining $v_p(0) := \infty$, the map $v_p \colon \qq \to \zz \cup \{\infty\}$ is the $p$-adic valuation map).

\medskip
%%%%%%%%%%%%%%%%
\subsection{Commutative Monoids}

We recall that a monoid is a semigroup with an identity element. Throughout this paper, identity elements are required to be inherited by submonoids and preserved by monoid homomorphisms. Moreover, we will tacitly assume that all monoids we will deal with are commutative and additively written. Let $M$ be a monoid. We set $M^\bullet := M \setminus \{0\}$, and we say that $M$ is \emph{trivial} if $M = \{0\}$. The monoid $M$ is called \emph{cancellative} if for all $a,b,c \in M$, the equality $a+b = a+c$ implies that $b=c$. Also, $M$ is called \emph{torsion-free} if for all $b,c \in M$ and $n\in \nn$, the equality $nb = nc$ implies $b=c$. The group of invertible elements of~$M$ is denoted by $\uu(M)$, and $M$ is called \emph{reduced} if the only invertible element of $M$ is $0$. The quotient $M/\uu(M)$ is a monoid that is called the \emph{reduced monoid} of $M$ and is denoted by $M_{\text{red}}$. 
\smallskip

The group $\gp(M)$ consisting of all the formal differences of elements of $M$ (under the operation naturally extended from that of $M$) is called the \emph{Grothendieck group} of $M$. When $M$ is cancellative, it can be minimally embedded into its Grothendieck group and this embedding is minimal in the following sense: $\gp(M)$ is the unique abelian group up to isomorphism such that any abelian group containing an isomorphic copy of $M$ will also contain an isomorphic copy of $\gp(M)$. The \emph{rank} of a cancellative monoid $M$ is defined to be the rank of $\gp(M)$ as a $\zz$-module or, equivalently, the dimension of the $\qq$-vector space $\qq \otimes_\zz \gp(M)$. Thus, the rank of a cancellative monoid gives a sense of its size or, more accurately, the size of the smallest vector space that contains one of its isomorphic copies. It follows from~\cite[Section~24]{lF70} and \cite[Theorem~2.9]{rG84} that a cancellative torsion-free monoid has rank $1$ if and only if it is isomorphic to an additive submonoid of $\qq$. The additive submonoids of $\qq$ that are not nontrivial groups are called \emph{Puiseux monoids} and have been actively investigated during the past decades. They will be helpful in this paper to provide the most significant (counter)examples we need.
\smallskip

Let $S$ be a subset of~$M$. We let $\langle S \rangle$ denote the smallest submonoid of $M$ containing ~$S$, and we call~$\langle S \rangle$ the submonoid of~$M$ \emph{generated} by~$S$. If $M = \langle S \rangle$, then $S$ is called a \emph{generating set} of~$M$, and~$M$ is called \emph{finitely generated} provided that $M$ has a finite generating set. Observe that any finitely generated Puiseux monoid is isomorphic to an additive (co-finite) submonoid of $\nn_0$ (additive co-finite submonoids of $\nn_0$ are called \emph{numerical monoids}).
\smallskip

For $b,c \in M$, we say that $c$ (\emph{additively}) \emph{divides} $b$ if $b = c+d$ for some $d \in M$, in which case we write $c \mid_M b$. A submonoid $N$ of $M$ is said to be \emph{divisor-closed} provided that the only pairs $(b,c) \in N \times M$ with $c \mid_M b$ are those with $c \in N$. %c \in M$ and $b \in N$ such that $c \mid_M b$ we have that $c \in N$. 
A \emph{maximal common divisor} (MCD) of a nonempty subset $S$ of~$M$ is a common divisor $d \in M$ of $S$ such that the only common divisors of the set $\{s-d : s \in S\}$ are the invertible elements of~$M$. The monoid $M$ is called an \emph{MCD-monoid} provided that every nonempty finite subset of~$M$ has an MCD. Also, for $k \in \nn$, we say that $M$ is a $k$-\emph{MCD-monoid} if every subset of~$M$ with cardinality $k$ has a maximal common divisor. Observe that every monoid is a $1$-MCD-monoid, while a monoid is an MCD-monoid if and only if it is a $k$-MCD-monoid for every $k \in \nn$. The notion of a $k$-MCD monoid seems to be introduced by Roitman in~\cite{mR93}.

\smallskip
%%%%%%%%%%%%%%%%%%%%%%%%%%%%%%%%%%%%%%%%%%%%%%%%%%%%%%%%%%%%%%%
\subsection{Atomicity and Ascending Chains of Principal Ideals}

An element $a \in M \! \setminus \! \uu(M)$ is called an \emph{atom} (or \emph{irreducible}) if whenever $a = b+c$ for some $b,c \in M$, then either $b \in \uu(M)$ or $c \in \uu(M)$. The set of atoms of $M$ is denoted by $\mathcal{A}(M)$. An element $b \in M$ is called \emph{atomic} if either $b \in \uu(M)$ or $b$ can be written as a sum of finitely many atoms (allowing repetitions). As coined by Cohn~\cite{pC68}, the monoid $M$ is \emph{atomic} if every element of $M$ is atomic. Following the more recent paper~\cite{nLL19} by Lebowitz-Lockard, we say that $M$ is \emph{nearly atomic} if there exists $c \in M$ such that $b+c$ is atomic for all $b \in M$. It follows directly from the definitions that every atomic monoid is nearly atomic. Following Boynton and Coykendall~\cite{BC15}, we say that the monoid $M$ is \emph{almost atomic} (resp., \emph{quasi-atomic}) provided that for each $b \in M$, there exists an atomic element (resp., an element) $c \in M$ such that $b+c$ is atomic. One can verify that every nearly atomic monoid is almost atomic, and it follows directly from the definitions that every almost atomic monoid is quasi-atomic.
\smallskip

A subset $I$ of $M$ is said to be an \emph{ideal} of~$M$ provided that $I + M := \{b+c : b \in I \text{ and } c \in M\} = I$ (or, equivalently, $I + M \subseteq I$). An ideal $I$ is \emph{principal} if the equality $I = b + M$ holds for some $b \in M$. An element $b \in M$ is said to satisfy the \emph{ascending chain condition on principal ideals} (ACCP) if every ascending chain of principal ideals of $M$ containing the ideal $b+ M$ stabilizes. The monoid $M$ is said to satisfy the \emph{ACCP} if every element of $M$ satisfies the ACCP. An ascending chain of principal ideals of $M$ is said \emph{to start} at an element $b \in M$ if the first ideal in the chain is $b+M$. Two ideal-theoretical notions weaker than the ACCP were introduced in~\cite{GL23}: the quasi-ACCP and the almost ACCP. We say that $M$ satisfies the \emph{almost ACCP} (resp., \emph{quasi-ACCP}) if for any nonempty finite subset $S$ of~$M$, there exists an atomic common divisor (resp., a common divisor) $d \in M$ of $S$ such that for some $s \in S$ the element $s-d$ satisfies the ACCP. It follows directly from the previous definitions that every monoid that satisfies the almost ACCP also satisfies the quasi-ACCP.
\smallskip

It is well known that every cancellative monoid that satisfies the ACCP is atomic (see~\cite[Proposition~1.1.4]{GH06}), and so every cancellative monoid that satisfies the ACCP also satisfies the almost ACCP. However, not every atomic monoid satisfies the ACCP, and several examples of cancellative monoids and integral domains witnessing this observation can be found in recent papers, including~\cite{GL23}. Moreover, it follows from \cite{GL23} that every monoid satisfying the almost ACCP is atomic (in Section~\ref{sec:atomicity and MCD}, we will construct a new rank-$1$ atomic monoid that does not satisfy the almost ACCP). The property of satisfying the quasi-ACCP does not imply that of being atomic as illustrated by the simple Puiseux monoid~$\qq_{\ge 0}$.

\medskip
%%%%%%%%%%%%%%%%%%%%%%%%%%%%%%%%%%%%%%%%%%%%%%%%%%%%
\subsection{Irreducible Divisors and Factorizations}

We say that $M$ is a \emph{Furstenberg monoid} or satisfies the \emph{Furstenberg property} if every non-invertible element of $M$ is divisible by an atom. The Furstenberg property was introduced by Clark in~\cite{pC17}. It follows from the definitions that every atomic monoid is a Furstenberg monoid. Following Grams and Warner~\cite{GW75}, we say that $M$ is an \emph{IDF-monoid} if every element of $M$ is divisible by only finitely many atoms up to associates (two elements of $M$ are \emph{associates} if their differences belong to $\uu(M)$). Then we say that $M$ is a \emph{TIDF-monoid} provided that it is a Furstenberg IDF-monoid. The TIDF (tightly irreducible divisor finite) property was introduced and first investigated by Zafrullah and the second author in~\cite{GZ23}.
\smallskip

Now assume that the monoid $M$ is atomic, which is equivalent to the fact that the reduced monoid $M_{\text{red}}$ is atomic. Let $\mathsf{Z}(M)$ be the free (commutative) monoid on the set of atoms $\mathcal{A}(M_{\text{red}})$. The elements of $\mathsf{Z}(M)$ are called \emph{factorizations}. Let $\pi \colon \mathsf{Z}(M) \to M_\text{red}$ be the unique monoid homomorphism fixing the set $\mathcal{A}(M_{\text{red}})$. For any $b \in M$, we set $\mathsf{Z}(b) := \pi^{-1}(b + \mathcal{U}(M))$ and call the elements of $\mathsf{Z}(b)$ (\emph{additive}) \emph{factorizations} of~$b$. If $|\mathsf{Z}(b)| < \infty$ for every $b \in M$, then $M$ is called a \emph{finite factorization monoid} (or an FFM for short). It follows from \cite[Theorem~2]{fHK92} that a monoid is an FFM if and only if it is an atomic IDF-monoid. In particular, every FFM is a TIDF-monoid. It follows from \cite[Corollary~1.4.4]{GH06} that every FFM satisfies the ACCP.
\smallskip

\medskip
%%%%%%%%%%%%%%%%%%%%%%%%%%%%%%%%%%%%%%%%%%%%%%%%
\subsection{Linearly Ordered Groups and Monoids}

The class consisting of linearly ordered monoids contains all Puiseux monoids and plays a fundamental role in this paper. The monoid $M$ is called \emph{linearly ordered} with respect to a total order relation $\preceq$ on $M$ if $\preceq$ is \emph{compatible} with the operation of $M$, which means that for all $b,c,d \in M$ the order relation $b \prec c$ ensures that $b+d \prec c+d$. %An abelian group that is linearly ordered with respect to a total order relation $\preceq$ is also called a \emph{linearly ordered abelian group} with respect to $\preceq$. 
We say that the monoid~$M$ is \emph{linearly orderable} provided that $M$ is a linearly ordered monoid with respect to some total order relation on~$M$. %An abelian group that is linearly orderable as a monoid is called a \emph{linearly orderable abelian group}. 
More than a century ago, it was proved by Levi~\cite{fL13} that every torsion-free abelian group is a linearly orderable monoid (or, simply, linearly orderable). % Also, one can readily check that every linearly orderable monoid (and so abelian group) must be torsion-free. 
From this Levi's result one can deduce the following well-known theorem. % that a monoid is linearly orderable if and only if it is cancellative and torsion-free. 

\begin{theorem} \label{thm:Levi's consequence}
    A monoid is linearly orderable if and only if it is cancellative and torsion-free.
\end{theorem}
\smallskip

Let $G$ be a linearly ordered abelian group (additively written) with respect to a total order relation~$\preceq$. The \emph{nonnegative cone} of $G$ is the submonoid $G^+$ of $G$ consisting of all nonnegative elements; that is,
\[
    G^+ := \{g \in G : 0 \preceq g \}.
\]
A submonoid of $G^+$ is called a \emph{positive submonoid} of $G$. In general, the monoid $M$ is called a \emph{positive monoid} provided that $M$ is isomorphic to a submonoid of the nonnegative cone of a linearly ordered abelian group. Observe that, as a consequence of Theorem~\ref{thm:Levi's consequence}, if the monoid $M$ is cancellative, reduced, and torsion-free, then its Grothendieck group $\gp(M)$ can be turn into a linearly ordered monoid so that~$M$ is a positive monoid of $\gp(M)$.
\smallskip

For $g \in G$, we set $|g| := \max\{\pm g \}$. For $g,h \in G$, we write $g = \textbf{O}(h)$ whenever $|g| \preceq n|h|$ for some $n \in \nn$. Now consider the equivalence relation $\sim$ on~$G$ obtained as follows: for $g,h \in G$, write $g \sim h$ whenever both equalities $g = \textbf{O}(h)$ and $h = \textbf{O}(g)$ hold. Set $\Gamma_G := (G \setminus \{0\})/\!\sim$ and consider the quotient map $v \colon G \setminus \{0\} \to \Gamma_G$. Then the binary relation $\le$ on $\Gamma$ defined by writing $v(g) \le v(h)$ for any $g,h \in G \setminus \{0\}$ such that $h = \textbf{O}(g)$ is a total order relation. The elements of $\Gamma_G$ are called \emph{Archimedean classes} of~$G$, and the quotient map $v$ is called the \emph{Archimedean valuation} on $G$. The group $G$ is called \emph{Archimedean} provided that $\Gamma_G$ is a singleton. A monoid is called \emph{Archimedean} if it is a positive monoid of an Archimedean group. According to one of the well-known H\"older's theorems, a linearly orderable abelian group is Archimedean if and only if it is order-isomorphic to a subgroup of the additive group~$\rr$.

\smallskip
%%%%%%%%%%%%%%%%%%%%%%%%%%%%%%%%%%
\subsection{Finitary Power Monoid}

Let $M$ be a monoid. We let $\Pfin(M)$ denote the monoid consisting of all nonempty finite subsets of $M$ under the so-called sumset operation: for any nonempty finite subsets $S$ and $T$ of $M$,
\[
	S + T := \{ s + t : (s,t) \in S \times T\}.
\]
The monoid $\Pfin(M)$ is called the \emph{finitary power monoid} of $M$. To simplify notation, in the scope of this paper we call the monoid $\Pfin(M)$ the \emph{power monoid} of $M$\footnote{In general, the power monoid of $M$ is the larger monoid consisting of all nonempty subsets of $M$ under the same sumset operation.}.

We say that the monoid $M$ is \emph{unit-cancellative} provided that for all $b,c \in M$ the equality $b + c = b$ implies that $c \in \uu(M)$. It follows from~\cite[Proposition~3.5]{FT18} that if $M$ is a linearly orderable monoid, then $\Pfin(M)$ is a unit-cancellative monoid. On the other hand, it is worth emphasizing that power monoids are extremely non-cancellative in the sense that $\Pfin(M)$ is cancellative if and only if the monoid $M$ is trivial. We proceed to prove some preliminary results about power monoids that we will need in the coming sections.

\begin{lemma} \label{lem:min and max} %\label{lem: shimon 1}
	Let $M$ be a linearly ordered monoid. For any $A,B,C \in \mathcal{P}_\emph{fin}(M)$ with $A+B = C$, the following statements hold:
	\[
		\min A + \min B = \min C \quad \text{ and } \quad \max A + \max B = \max C.
	\]
\end{lemma}

\begin{proof}
	We only verify that $\min A + \min B = \min C$ as the other identity follows similarly. Let $\preceq$ be the total order relation under which $M$ is a linearly ordered monoid. Since $\min C$ belongs to $C$ and $C = A+B$, we can take some $a \in A$ and $b \in B$ such that $a+b = \min C$. As $\min A \preceq a$ and $\min B \preceq b$, $\min A+\min B \preceq \min C$. On the other hand, the fact that $\min A + \min B \in A+B = C$ ensures that $\min C \preceq \min A + \min B$. %\prec \min C$, then $\min C$ would not be minimal in $C$. Thus, strict equality is necessary, yielding $\min A + \min B = \min C$.
\end{proof}

The following corollary is an immediate consequence of Lemma~\ref{lem:min and max}.

\begin{cor}  \label{cor:divisibility from min and max} %\label{cor: shimon 1}
	Let $M$ be a linearly ordered monoid. If $A, B \in \mathcal{P}_\emph{fin}(M)$, then $A \mid_{\mathcal{P}_\emph{fin}(M)} B$ implies that $\min A\mid_M \min B$.
\end{cor}

The following lemma will also be helpful later.

\begin{lemma} \label{lem: shimon 2}
	Let $M$ be a linearly ordered monoid. For $A,B \in \mathcal{P}_\emph{fin}(M)$ such that $A \mid_{\mathcal{P}_\emph{fin}(M)} B$, if $\min A = \min B$, then either $A=B$ or $|A|<|B|$.
\end{lemma}

\begin{proof}
	Take $A,B \in \mathcal{P}_\text{fin}(M)$ such that $A \mid_{\mathcal{P}_\text{fin}(M)} B$, and assume that $\min A = \min B$. If $A=B$, then we are done. Therefore assume that $A\neq B$. Since $A \mid_{\mathcal{P}_\text{fin}(M)} B$, we can take $D \in \mathcal{P}_\text{fin}(M)$ such that $A+D = B$. By Lemma~\ref{lem:min and max}, the equality $\min A + \min D = \min B$ holds. Hence the equality $\min A = \min B$ implies that $\min D = 0$. As a result, $A = A + \{0\} \subseteq A + D = B$, and so the inequality $|A| < |B|$ follows from the fact that $A \neq B$.
\end{proof}

% \color{red}
% TODO: Bring a simple counterexample for the non-cancellative case.
% \color{\black}

\bigskip
%%%%%%%%%%%%%%%%%%%%%%%%
%%%%%%%%%%%%%%%%%%%%%%%%
\section{Ascending Chains of Principal Ideals}
\label{sec:AC of PI}

It is known that if a linearly orderable monoid $M$ satisfies the ACCP, then the power monoid of $M$ also satisfies the ACCP. In this section, we will establish parallel ascent results for the quasi-ACCP and the almost ACCP. Before proving the ascent of these two properties, we need the following preliminary known lemma (we include a proof here for the sake of completeness).

\begin{lemma} \label{lem:size of the sum}
	Let $M$ be a linearly orderable monoid. For any $S,T \in \mathcal{P}_\emph{fin}(M)$, the following statements hold.
	\begin{enumerate}
		\item $|S+T| \ge |S| + |T| - 1 \ge \max\{|S|,|T|\}$.
		\smallskip
		
		\item If $|S| \ge 2$, then $|S+T| > |T|$.
	\end{enumerate}
\end{lemma}

\begin{proof}
	(1) Take $S,T \in \mathcal{P}_\text{fin}(M)$. The first inequality $|S+T| \ge |S| + |T| - 1$ is \cite[Proposition~3.5]{FT18}, and the second inequality follows immediately.
    %It suffices to show that $|S+T| \ge |T|$. To do so, fix $s \in S$, and note that the fact that $M$ is a linearly orderable monoid implies that it is cancellative, and so for any element $s \in S$ the sets $\{s\} + T$ and $T$ have the same cardinality. Thus, $\{s\} + T$ is a subset of $S + T$, which implies that $|S+T| \ge |\{s\} + T| = |T|$.  
	\smallskip
	
	(2) %Set $\ppp := \mathcal{P}_\text{fin}(M)$, and 
    Let $\preceq$ be a total order relation on $M$ turning $M$ into a linearly ordered monoid. Set $s := \min S$ and $t := \min T$. Because $|S| \ge 2$, we can take $r \in S \setminus \{s\}$. Now observe that $s+t \prec r+t = \min( \{r\} + T )$, so $\{s+t\} \notin \{r\}+T$. This means that $|\{s+t\} \cup (\{r\} + T)| = |T| + 1$. Finally, the inclusion $\{s+t\} \cup (\{r\}+T) \subseteq S+T$, along with the fact that $T$ and $\{r\}+T$ have the same cardinality, guarantees that $|S+T| \ge |T|+1 > |T|$.
\end{proof}

For any monoid $M$, it is clear that the set of all singletons in $M$ is a submonoid of $\Pfin(M)$. In light of the second part of the previous lemma, we obtain that such a submonoid is divisor-closed. We record this easy remark here for future reference.

\begin{cor}
    Let $M$ be a linearly orderable monoid. Then $\{S \in \Pfin(M) : |S|=1 \}$ is a divisor-closed submonoid of $\Pfin(M)$.
\end{cor}

We are in a position to establish the main result of this section, the ascent of the quasi-ACCP and the almost ACCP to power monoids in the class of linearly orderable monoids.

\begin{theorem}
	Let $M$ be a linearly orderable monoid. Then the following statements hold.
	\begin{enumerate}
		\item If $M$ satisfies the quasi-ACCP, then $\mathcal{P}_\emph{fin}(M)$ also satisfies the quasi-ACCP.
		\smallskip
		
		\item If $M$ satisfies the almost ACCP, then $\mathcal{P}_\emph{fin}(M)$ also satisfies the almost ACCP.
	\end{enumerate}
\end{theorem}

\begin{proof}
	To make our notation less cumbersome, we write $\ppp$ instead of $\mathcal{P}_\text{fin}(M)$. 
	\smallskip
	
	(1) Assume that $M$ satisfies the quasi-ACCP. In order to argue that $\ppp$ also satisfies the quasi-ACCP, fix a nonempty finite subset $\{S_1, \ldots, S_n\}$ of~$\ppp$. Now set $S := S_1 \cup \cdots \cup S_n$. Because $S$ is a nonempty finite subset of $M$, the fact that $M$ satisfies the quasi-ACCP allows us to pick a common divisor $d \in M$ of $S$ and also an element $s \in S$ such that $s-d$ satisfies the ACCP in~$M$. Take an index $j \in \ldb 1,n \rdb$ such that $s \in S_j$. For each $i \in \ldb 1, n \rdb$, the inclusion $S_i \subseteq S$ ensures that $d$ is a common divisor of $S_i$ in $M$, and so $\{d\} \mid_\ppp S_i$. 
	
	Thus, it suffices to show that $S_j - \{d\}$ satisfies the ACCP in $\ppp$. To do this, take an ascending chain $(B_n + \ppp)_{n \ge 0}$ of principal ideals of $\ppp$ starting at $S_j - \{d\}$ and, as $B_0$ and $S_j - \{d\}$ are associates, we can assume that $B_0 = S_j - \{d\}$. Now set $b_0 := s-d$ and take $b_1 \in B_1$ such that $b_1 \mid_M b_0$, and then note that if $b_0, \dots, b_n$ are elements in $M$ such that $b_i \in B_i$ and $b_i \mid_M b_{i-1}$ for every $i \in \ldb 1,n \rdb$, then the fact that $B_{n+1} \mid_\ppp B_n$ allows us to take $b_{n+1} \in B_{n+1}$ such that $b_{n+1} \mid_M b_n$. Hence we have inductively constructed a chain $(b_n + M)_{n \ge 0}$ of principal ideals of $M$ with $b_0 = s-d$ such that $b_n \in B_n$ for every $n \in \nn_0$. Since the ascending chain $(b_n + M)_{n \ge 0}$ starts at $s-d$, which is an element satisfying the ACCP in $M$, there exists an index $k_1 \in \nn$ such that whenever $n > k_1$ the equality $b_n + M = b_{n-1} + M$ holds, and so $b_{n-1} - b_n \in \uu(M)$. On the other hand, it follows from part~(1) of Lemma~\ref{lem:size of the sum} that, for each $n \in \nn_0$, the divisibility relation $B_{n+1} \mid_\ppp B_n$ implies that $|B_n| \ge |B_{n+1}|$, whence there exists an index $k_2 \in \nn$ such that $|B_n| = |B_{k_2}|$ for every $n \ge k_2$. Now, for each $n \in \nn$, take a subset $C_n$ of $M$ such that $B_{n-1} = B_n + C_n$. Then by part~(2) of Lemma~\ref{lem:size of the sum}, for each $n \in \nn$ with $n > k_2$ the set $C_n$ must be a singleton, and so $C_n = \{b_{n-1} - b_n\}$. As a consequence, for each $n \in \nn$ with $n > \max\{k_1, k_2\}$, we obtain that the element $C_n$ of $\ppp$ is the singleton containing the invertible element $b_{n-1} - b_n$, and so the chain $(B_n + \ppp)_{n \ge 0}$ must stabilize. Thus, we conclude that $S_j - \{d\}$ satisfies the ACCP.
	\smallskip
	
	(2) Suppose now that $M$ satisfies the almost ACCP. As in the previous part, fix a nonempty finite subset $\{S_1, \dots, S_n\}$ of $\ppp$, and use the fact that $M$ satisfies the almost ACCP to find an atomic common divisor $d \in M$ of $S_1 \cup \cdots \cup S_n$ such that $s-d$ satisfies the ACCP in~$M$. If $d$ is an invertible element of $M$, then $\{d\}$ is an invertible element of $\ppp$. Otherwise, we can write $d = a_1 + \dots + a_\ell$ for some $a_1, \dots, a_\ell \in \mathcal{A}(M)$, in which case, $\{a_1\}, \dots, \{a_\ell\} \in \mathcal{A}(M)$ and so $\{d\}$ can be written as a sum of atoms in $\ppp$, namely, $\{d\} = \{a_1\} + \dots + \{a_\ell\}$. Hence $\{d\}$ must be an atomic element in $\ppp$. Finally, proceeding \emph{mutatis mutandis} as we did in part~(1), we can show that $S_j - \{d\}$ satisfies the ACCP in $\ppp$, where $j$ is an index in $\ldb 1,n \rdb$ such that $d \in S_j$. Hence we conclude that the power monoid $\ppp$ also satisfies the almost ACCP.
\end{proof}

\bigskip
%%%%%%%%%%
%%%%%%%%%%
\section{Atomicity and Maximal Common Divisors}
\label{sec:atomicity and MCD}

In this first section, we will investigate the ascent of atomicity as well as the ascent of some weaker notions of atomicity from linearly ordered monoids to their corresponding power monoids.

\medskip
%%%%%%%%%%%%%%%%%%%%%%
\subsection{Existence of Maximal Common Divisors}

Our next goal is to prove that, for any linearly orderable monoid $M$, the power monoid $\Pfin(M)$ is atomic if and only if $M$ is an atomic MCD-monoid. This result not only generalizes but also strengthens~\cite[Theorem~3.2]{GLRRT24}, in the sense that it gives a complete characterization of atomic power monoids of linearly ordered monoids (which are more general than Puiseux monoids). Given that the existence of maximal common divisors (MCDs) is essential for our characterization, let us first establish the following proposition, which essentially states that a linearly orderable monoid is an MCD-monoid if and only if its power monoid is an MCD-monoid.

\begin{prop} \label{prop:power monoid MCD}
For a linearly orderable monoid $M$, the following conditions are equivalent.
    \begin{enumerate}
        \item[(a)] $M$ is an MCD-monoid.
        \smallskip
        
        \item[(b)] $\mathcal{P}_\emph{fin}(M)$ is an MCD-monoid.
        \smallskip
    
        \item[(c)] $\mathcal{P}_\emph{fin}(M)$ is a $k$-MCD-monoid for some $k \in \mathbb{N}_{\ge 2}$.
    \end{enumerate}
\end{prop}

\begin{proof}
    (a) $\Rightarrow$ (b): Assuming that $M$ is an MCD-monoid, let us show that each nonempty finite subset $\mathcal{S}$ of $\Pfin(M)$ has an MCD by using induction on $\sum_{S \in \mathcal{S}} |S|$. For the base case, note that if, for a nonempty finite subset $\mathcal{S}$ of $\Pfin(M)$, the equality $\sum_{S \in \mathcal{S}} |S| = 1$ holds, then $\mathcal{S} = \br{\br{m}}$ for some $m \in M$, which implies that $\br{m}$ is an MCD of $\mathcal{S}$ in $\Pfin(M)$. Now fix $n \in \nn$ such that every nonempty finite subset $\mathcal{S}$ of $\Pfin(M)$ with $\sum_{S \in \mathcal{S}} |S| \le n$ has an MCD in $\Pfin(M)$. Let $\mathcal{T}$ be a nonempty finite subset of $\Pfin(M)$ with $\sum_{T \in \mathcal{T}} |T| = n+1$, and let us argue that $\mathcal{T}$ has an MCD in $\Pfin(M)$. Consider the following two cases.
    \smallskip
    
    \textsc{Case 1:} $\mathcal{T}$ has a common divisor in $\Pfin(M)$ that is not a singleton. Let $D$ be a common divisor of $\mathcal{T}$ in $\Pfin(M)$ such that $|D| \ge 2$. Then we can write $\mathcal{T} = D + \mathcal{T}'$ for some $\mathcal{T}'$ in $\Pfin(M)$, in which case the inequality $\sum_{T \in \mathcal{T}'} |T| < \sum_{T \in \mathcal{T}} |T|$ holds in light of Lemma~\ref{lem:size of the sum}. Therefore our induction hypothesis ensures the existence of an MCD $T'$ of $\mathcal{T}'$ in $\Pfin(M)$. It immediately follows now that $D + T'$ is an MCD of $\mathcal{T}$ in $\Pfin(M)$.
    \smallskip
    
    \textsc{Case 2:} Each common divisor of $\mathcal{T}$ in $\Pfin(M)$ is a singleton. Because $M$ is an MCD-monoid, the finite nonempty subset $U := \bigcup_{T \in \mathcal{T}} T$ of $M$ must have an MCD, namely, $m_0 \in M$. As $m_0$ is a common divisor of $T$ in $M$ for all $T \in \mathcal{T}$, it follows that $\{m_0\}$ is a common divisor of $\mathcal{T}$. To argue that $\{m_0\}$ is indeed an MCD of $\mathcal{T}$, take a nonempty finite subset $D'$ of $M$ such that $\{m_0\} + D'$ is a common divisor of $\mathcal{T}$. By assumption, $\{m_0\} + D'$ is a singleton, and so $D' = \{d\}$ for some $d \in M$. Now the fact that $m_0$ is an MCD of $U$ ensures that $d \in \uu(M)$. Hence $\{m_0\}$ is an MCD of $\mathcal{T}$ in $\Pfin(M)$.
    \smallskip

    (b) $\Rightarrow$ (c): This is straightforward.
    \smallskip

    (c) $\Rightarrow$ (a): Assume that $\Pfin(M)$ is a $k$-MCD monoid for some $k \in \NN_{\ge 2}$, and so a $2$-MCD-monoid. To show that $M$ is an MCD-monoid, it suffices to fix a finite subset $S$ of $M$ with $|S| \ge  2$ and prove that $S$ has an MCD in $M$. Take $s_0 \in S$, and then set $\mathcal{T} := \br{\br{s_0}, S \setminus \br{s_0}}$, which is a subset of $\Pfin(M)$. Because $\mathcal{T}$ contains a singleton, it follows from Lemma~\ref{lem:size of the sum} that all common divisors of $\mathcal{T}$ in $\Pfin(M)$ are singletons. In addition, as $\Pfin(M)$ is a $2$-MCD-monoid and $|\mathcal{T}| = 2$, we can take $m_0 \in M$ such that $\br{m_0}$ is an MCD of $\mathcal{T}$ in $\Pfin(M)$. This immediately implies that $m_0$ is an MCD of $S$ in $M$.
\end{proof}

For a monoid $M$, we say that an element $S$ of $\Pfin(M)$ is \emph{indecomposable} if whenever we can write $S = U+V$ for some $U, V \in \Pfin(M)$ either $|U| = 1$ or $|V| = 1$. In order to establish the primary result of this section, we need the following lemma.

\begin{lemma} \label{lem:indecomposable elements}
	Let $M$ be a linearly orderable monoid, and let $S$ be a finite subset of $M$ with $|S| \ge 2$. Then the following statements hold.
	\begin{enumerate}
		\item If $\min S + \max S > 0$, then $S \cup \{4 \max S\}$ is indecomposable.
		\smallskip
		
		\item If $\min S + \max S < 0$, then $S \cup \{4 \min S\}$ is indecomposable.
	\end{enumerate}
\end{lemma}

\begin{proof}
	Take $s_1, \dots, s_n \in M$ with $s_1 < \dots < s_n$ such that $S = \{s_1, \dots, s_n\}$.
	\smallskip
	
	(1) Assume that $\min S + \max S \ge 0$, and set $T := S \cup \{4s_n\}$. From $s_1 + s_n > 0$, we obtain that $s_n > 0$. Assume, towards a contradiction, that we can pick $U$ and $V$ to be non-singletons nonempty finite subsets of $M$ such that $T = U + V$ in $\Pfin(M)$. Now set $u_1 := \max U$ and then set $u_2 := \max (U \setminus \{u_1\})$. Similarly, set $v_1 := \max V$ and then set $v_2 := (V \setminus \{v_1\})$. Therefore we see that $4 s_n = \max T = \max (U+V) = u_1 + v_1$ and
	\[
		s_n = \max (T \setminus \{4s_n\}) = \max \big( (U + V) \setminus \{u_1 + v_1\} \big) \in \{u_1 + v_2, u_2 + v_1\}.
	\]
	Assume, without loss of generality, that $s_n = u_2 + v_1$. Then $(u_1 + v_1) - (u_2 + v_1) = 3s_n$ and, therefore, $u_1 = u_2 + 3s_n$. Similarly, it follows from $s_n \ge u_1 + v_2$ that $(u_1 + v_1) - (u_1 + v_2) \ge 3s_n$, and so $v_2 + 3s_n \le v_1$. Thus, $u_2 + v_2 + 6s_n \le u_1 + v_1 = 4s_n$, which implies that $u_2 + v_2 + 2 s_n \le 0$. On the other hand, we find that
	\[
		u_2 + v_2 + 2s_n \ge \min U + \min V + 2 s_n \ge (s_1 + s_n) + s_n \ge s_n > 0,
	\]
	which contradicts the inequality $u_2 + v_2  + 2s_n  \le 0$. As a consequence, $T$ is an indecomposable element of $\Pfin(M)$. %as the sum of two non-singleton
	\smallskip
	
	(2) Now assume that $\min S + \max S < 0$. This part is symmetric to part~(1): indeed, after setting $T := S \cup \{4s_1\}$, we can take the elements $u_1$ and $u_2$ (resp., $v_1$ and $v_2$) to be the smallest element and second smallest element of $U$ (resp., $V$), and then we can repeat the argument already given in the proof of the previous part.
\end{proof}

We are in a position to characterize the atomic power monoids of linearly orderable monoids.

\begin{theorem} \label{thm:power monoid atomic}
    For any linearly orderable monoid $M$, the following conditions are equivalent.
    \begin{enumerate}
        \item[(a)] $M$ is an atomic MCD-monoid.
        \smallskip
        
        \item[(b)] $\mathcal{P}_\emph{fin}(M)$ is an atomic monoid.
        \smallskip
        
        \item[(c)] $\mathcal{P}_\emph{fin}(M)$ is an atomic MCD-monoid.
        \smallskip
        
        \item[(d)] $\mathcal{P}_\emph{fin}(M)$ is an atomic $k$-MCD-monoid for some $k\in\mathbb{N}_{\ge 2}$.
    \end{enumerate}
\end{theorem}

\begin{proof}
    Let $M$ be a linearly orderable monoid, and set $\mathcal{M} := \{ \br{m} : m \in M \}$, which is a divisor-closed submonoid of $\Pfin(M)$.
    \smallskip
    
    (a) $\Rightarrow$ (c): The proof that $\Pfin(M)$ is an atomic monoid follows the line of the proof of \cite[Theorem~3.2]{GLRRT24} \emph{mutatis mutandis}. Then, it follows from Proposition~\ref{prop:power monoid MCD} that $\Pfin(M)$ is an atomic MCD-monoid.
    \smallskip

    (c) $\Rightarrow$ (d): This is straightforward.
    \smallskip

    (d) $\Rightarrow$ (b): This is also straightforward.
    \smallskip

    (b) $\Rightarrow$ (a): Because $\mathcal{M}$ is a divisor-closed submonoid of $\Pfin(M)$, from the fact that $\Pfin(M)$ is atomic one obtains that $\mathcal{M}$ is atomic. As a consequence, $M$ is also atomic as it is naturally isomorphic to $\mathcal{M}$. %In order to prove that $M$ is an MCD-monoid, we need the following claim.
    \smallskip

    To prove that $M$ is an MCD-monoid, we fix a nonempty finite subset $S$ of $M$ and argue that~$S$ has an MCD in~$M$. As every singleton has an MCD in $M$, we can assume that $|S| \ge 2$. On the other hand, after replacing $S$ by a suitable subset, we can further assume that $s+t \neq 0$ for all $s,t \in S$. Take $s_1, \dots, s_n \in M$ with $s_1 < \dots < s_n$ such that $S = \{s_1, \dots, s_n\}$. %We split the rest of the proof into two cases depending on whether $s_1 + s_n > 0$ or $s_1 + s_n < 0$. 
%    \smallskip
%
%    \textsc{Case 1:} $s_1 + s_n > 0$. This clearly implies that $s_n > 0$. 
    Now set $T := S \cup \br{t}$, where $t := 4s_n$ if $s_1 + s_n > 0$ and $t := 4 s_1$ if $s_1 + s_n < 0$. 
    
    Since $\Pfin(M)$ is an atomic monoid, we can write $T = A_1 + \dots + A_\ell$ for some atoms $A_1, \dots, A_\ell$ of $\Pfin(M)$. In light of Lemma~\ref{lem:indecomposable elements}, we can assume that $A_i$ is a singleton if and only if $i \in \ldb 1, \ell - 1 \rdb$ (at least one of $A_1, \dots, A_\ell$ is not a singleton because $T$ is not a singleton). Write $A_i = \br{a_i}$ for every $i \in \ldb 1,\ell - 1 \rdb$. Thus, the fact that $\{a_1\}, \dots, \{a_{\ell-1}\}$ are atoms of $\Pfin(M)$ implies that they are also atoms of the divisor-closed submonoid $\mathcal{M}$ of $\Pfin(M)$, whence the natural isomorphism between $\mathcal{M}$ and $M$ ensures that $a_1, \dots, a_{\ell-1} \in \mathcal{A}(M)$. Set $A := \{a_1 + \dots + a_{\ell - 1}\}$ and $\mathcal{T} := \{S, \{t\}\}$, and let us argue the following claim.
    \smallskip
    
    \noindent \textsc{Claim.} $A$ is an MCD of $\mathcal{T}$ in $\Pfin(M)$.
    \smallskip
    
    \noindent \textsc{Proof of Claim.} Since $A + A_\ell = T = \{s_1, \dots, s_n \} \cup \{t\}$ and $A$ is a singleton, we see that~$A$ divides both $S = \{s_1, \dots, s_n\}$ and $\{t\}$ in $\Pfin(M)$, so $A$ is a common divisor of $\mathcal{T}$. Now suppose that $A + D$ divides both $S$ and $\{t\}$ in $\Pfin(M)$ for some nonempty finite subset $D$ of $M$. Then~$D$ must be a singleton because it divides the singleton $\{t\}$ in $\Pfin(M)$. Because $A+D$ is a singleton and also a common divisor of $S$ and $\{t\}$ in $\Pfin(M)$, it follows that $A+D$ divides $T$ in $\Pfin(M)$. Now take a nonempty finite subset $D'$ of $M$ such that $(A + D) + D' = T = \{a_1 + \dots + a_{\ell - 1}\} + A_\ell$. This implies that $D+D' = A_\ell$ (every singleton is a cancellative element in $\Pfin(M)$). Now the fact that $A_\ell$ is an atom of $\Pfin(M)$ guarantees that either $D$ or $D'$ is invertible in $\Pfin(M)$. Since $D'$ is not a singleton (because $A_\ell$ is not a singleton), it follows that $D$ is invertible in $\Pfin(M)$. Hence we conclude that $A$ is an MCD of $\mathcal{T}$ in $\Pfin(M)$, and the claim is established.
    \smallskip
    
    It follows from the established claim that $s := a_1 + \dots + a_{\ell-1}$ is a common divisor of $S$ in $M$. To argue that $s$ is an MCD of $S$ in $M$, take $d \in M$ such that $s+d$ is a common divisor of $S$ in $M$. Then $A + \{d\}$ divides both $S$ and $\{t\}$ in $\Pfin(M)$, and so the established claim implies that $\{d\}$ is invertible in $\Pfin(M)$. Therefore $d \in \uu(M)$, and so $s$ is an MCD of $S$ in~$M$. We can now conclude that $M$ is an atomic MCD-monoid, thereby completing our proof.
\end{proof}

\medskip
%%%%%%%%%%%%%%%%%%%%%%%%%%%%%%%%%%%%
\subsection{Non-Ascent of Atomicity} 

Next we construct an atomic rank-$1$ torsion-free monoid~$M$ whose power monoid $\Pfin(M)$ is not atomic, confirming the result first given in~\cite[Section~3]{GLRRT24} that the property of being atomic does not ascend to power monoids. First, we argue the following lemma.

\begin{lemma} \label{lem:coefficients of an atom with unique negative p-adic valuation}
	Let $M$ be a Puiseux monoid generated by a set $S$, and let $(p,a)$ be the only pair in $\pp \times S$ such that $p \mid \mathsf{d}(s)$. Then the following statements hold.
	\begin{enumerate}
		\item $a \in \mathcal{A}(M)$.
		\smallskip
		
		\item For each $q \in M$, the following set is a singleton:
		\begin{equation} \label{eq:a singleton set}
			\big\{c + p\zz : q = c a + r \text{ for some } c \in \nn_0 \text{ and } r \in \big\langle S \setminus \{a\} \big\rangle \big\}.
		\end{equation}
	\end{enumerate}
\end{lemma}

\begin{proof}
	(1) This follows immediately as if $a \in \langle S \setminus \{a\} \rangle$, then we could take $c_1, \dots, c_\ell \in \nn$ and $s_1, \dots, s_\ell \in S \setminus \{a\}$ such that $a = c_1 s_1 + \dots + c_\ell s_\ell$, which is not possible because $v_p(a) \le -1$ while $v_p(c_1 s_1 + \dots + c_\ell s_\ell) \ge 0$.
	\smallskip
	
	(2) Set $N := \langle S \setminus \{a\} \rangle$, and observe that $v_p(r) \ge 0$ for every $r \in \gp(N)$. Now fix $q \in M$, and let $\mathcal{C}_q$ be the set described in~\eqref{eq:a singleton set}. Let $c_1 + p\zz$ and $c_2 + p \zz$ be two elements of $\mathcal{C}_q$, and then write $q = c_1 a + r_1 = c_2 a + r_2$ for some $r_1, r_2 \in N$. Since $(c_1 - c_2)a = r_2 - r_1 \in \gp(N)$, we see that
	\[
		v_p(c_1 - c_2) \ge 1 + v_p((c_1 - c_2)a) = 1 + v_p(r_2 - r_1) \ge 1,
	\]
	which means that $p \mid c_1 - c_2$ or, equivalently, $c_1 + p \zz = c_2 + p \zz$. As a consequence, we can conclude that $\mathcal{C}_q$ is a singleton for each $q \in M$.
\end{proof}

Based on part~(2) of Lemma~\ref{lem:coefficients of an atom with unique negative p-adic valuation}, we introduce the following notation.
\smallskip

\noindent \textbf{Notation.} Let the notation be as in Lemma~\ref{lem:coefficients of an atom with unique negative p-adic valuation}. For each $q \in M$, we let $c_{a,p}(q)$ denote the unique element of $\mathbb{Z}/p\mathbb{Z}$ contained in the singleton~\eqref{eq:a singleton set}, and we set $c_a(q)$ instead of $c_{a,p}(q)$ when $p$ is unique given $a$. This satisfies some useful properties: for instance, one can check that
\[
    c_{a,p}(q+r) = c_{a,p}(q) + c_{a,p}(r) \quad \text{for all} \quad q,r \in M.
\]
Also note that $c_{a,p}(q)$ is the same for any $S$ we choose where $a$ satisfies the desired condition. Thus, as long as $S$ exists, the notation is well defined.
\smallskip

We proceed to produce an example of an atomic Puiseux monoid that is not $2$-MCD. The following example, which is motivated by \cite[Example~3.3]{GLRRT24}, not only improves the choice of the generating set of the latter, but will also help us resolve the question of whether near atomicity ascends to power monoids in the class of linearly orderable monoids.

\begin{example} \label{ex:non-ascent of atomicity} (cf.~\cite[Example~3.3]{GLRRT24})
	Let $p_n$ be the $n^\text{th}$ prime in $\mathbb{P}_{\ge 5}$, and consider the Puiseux monoid $M$ generated by the set $A_1 \cup A_2$, where
	\[
		A_1 := \bigg\{ \frac{1}{2^n p_{2n+2}} : n \in \mathbb{N}_0 \bigg\} \quad \text{ and } \quad A_2 := \bigg\{ \frac1{p_{2n+1}} \bigg( \frac13 + \frac{1}{2^n} \bigg) : n \in \mathbb{N}_0 \bigg\}.
	\]
	By part~(1) of Lemma~\ref{lem:coefficients of an atom with unique negative p-adic valuation}, every element of $A_1 \cup A_2$ is an atom of $M$ and, therefore, $\mathcal{A}(M) = A_1 \cup A_2$. Hence $M$ is an atomic monoid. We will argue that $M$ is not a $2$-MCD monoid by showing that the subset $\big\{ 1, \frac{4}{3} \big\}$ of $M$ does not have an MCD (both $1$ and~$\frac43$ belong to~$M$).

	For $z \in \mathsf{Z}(M)$ and $a \in \mathcal{A}(M)$, we let $z_a$ be the number of copies of $a$ that appear in~$z$, and let $p_a$ be the unique prime in $\pp_{\ge 5}$ dividing $\textsf{d}(a)$. Observe that, for each $r \in M$, if the equality $v_{p_a}(r) = 0$ holds, then for any $z \in \mathsf{Z}(r)$ there exists $d_a \in \nn_0$ such that $z_a = d_a p_a$. %for some $d_a \in \nn_0$.
	
	We claim that $a \nmid_M 1$ for any $a \in A_2$. Suppose, by way of contradiction, that $a \mid_M 1$ for some $a \in A_2$, and fix $z \in \mathsf{Z}(1)$ with $z_a \ge 1$. The equality $v_{p_a}(1) = 0$ implies that $d_a := \frac{z_a}{p_a} \in \nn_0$, and so that $d_a \ge 1$. Now the fact that $1 \ge z_a a = d_a \big( \frac13 + \frac1{2^n} \big)$ for some $n \in \nn_0$ ensures that $d_a \in \{1,2\}$. Thus, one can deduce from $v_3(1) = 0$ that (as an atom of $A_2$ appears in $z$) there must exist pairwise distinct atoms $a_1, a_2, a_3 \in A_2$ such that at least two of them appear in the factorization $z$ and $d_{a_1} + d_{a_2} + d_{a_3} = 3k$ for some $k \in \nn$. As a consequence,
	\[
		\sum_{i=1}^3 z_{a_i} a_i \ge \sum_{i=1}^3d_{a_i}\bigg( \frac13 + \frac1{2^{n_i}}\bigg) > 1
	\]
	for some $n_1, n_2, n_3 \in \nn_0$, which contradicts that $1 \ge \sum_{i=1}^3 z_{a_i} a_i$ (as $\sum_{i=1}^3 z_{a_i} a_i$ is a factorization of a divisor of $1$ in~$M$). Therefore $a \nmid_M 1$ for any $a \in A_2$, as claimed.

	Let us prove now that $\big\{1, \frac43 \big\}$ has no MCD in $M$. Fix a common divisor $q \in M$ of $\big\{ 1, \frac43\big\}$, and let us find a positive common divisor of $\big\{1-q, \frac43 - q \big\}$ in~$M$. Since $q \mid_M 1$, it follows from the claim proved in the previous paragraph that $a \nmid_M q$ for any $a \in A_2$. Let $z$ be a factorization of $\frac43 - q$. Since $v_3\big( \frac43 -q \big) = -1$, the inclusion $v_3(A_1) \subseteq \nn_0$ implies that at least one atom from $A_2$ appears in~$z$. As the sum of any four atoms of $A_2$ is larger than $\frac43$, at most three pairwise distinct atoms of $A_2$ can appear in $z$. Let $\{a_1, a_2, a_3\}$ be a $3$-subset of $A_2$ containing the atoms of $A_2$ that appear in $z$. Then $v_3\big( \frac43 - (z_{a_1}a_1 + z_{a_2}a_2 + z_{a_3}a_3) \big) \ge 0$, and so
	\[
		v_3\bigg( \frac43 -  \frac{d_{a_1} + d_{a_2} + d_{a_3}}3 \bigg) = v_3\bigg( \frac43 - \sum_{i=1}^3d_{a_i} \bigg( \frac13 + \frac1{2^{n_i}} \bigg)\bigg) = v_3\bigg( \frac43 - \sum_{i=1}^3 z_{a_i} a_i \bigg) \ge 0
	\]
	for some $n_1, n_2, n_3 \in \nn_0$. Thus, $d_{a_1} + d_{a_2} + d_{a_3} \in 1 + 3\nn_0$, and so $d_{a_1} + d_{a_2} + d_{a_3} = 1$ because the inequality $d_{a_1} + d_{a_2} + d_{a_3}  \ge 4$ implies that $z_{a_1}a_1 + z_{a_2}a_2 + z_{a_3}a_3 > \frac43$. Therefore there is a unique index $n \in \nn_0$ such that the atom $a_n :=  \frac1{p_{2n+1}} \big( \frac13 + \frac{1}{2^n} \big) \in A_2$ appears in~$z$, and it appears exactly $p_{2n+1}$ times. Then we can take $r \in \langle A_1 \rangle$ such that $\frac43 - q = p_{2n+1} a_n + r = \big(\frac13 + \frac1{2^n} \big) + r$, which implies that
	\[
		1-q = \frac{1}{2^{n+1}} + \bigg( \frac1{2^{n+1}} + r \bigg) \quad \text{ and } \quad \frac43 - q =  \frac1{2^{n+1}} + \bigg( \frac13 + \frac1{2^{n+1}} \bigg) + r.
	\]
	Hence $\frac1{2^{n+1}}$ is a positive common divisor of the set $\big\{1 - q, \frac43 - q \big\}$ in $M$, and so we conclude that $\big\{1 - q, \frac43 - q \big\}$ does not have an MCD in $M$. Thus,~$M$ is not an MCD-monoid.
\end{example}

The following remark is an immediate consequence of Theorem~\ref{thm:power monoid atomic}.

\begin{remark}
    Although the monoid constructed in Example~\ref{ex:non-ascent of atomicity} is atomic, its power monoid is not atomic.
\end{remark}

\section{Notions Weaker than Atomicity}
\label{sec:notions weaker than atomicity}
In this section, we study three natural generalizations of atomicity recently studied in the literature: near atomicity, almost atomicity, and quasi-atomicity.

\medskip
%%%%%%%%%%%%%%%%%%%%%%%%%%%
\subsection{Near Atomicity}

%Then we provide an example of an atomic Puiseux monoid whose power monoid is not atomic. Although the ascent of atomicity to power monoids was previously resolved in~\cite[Section~3]{GLRRT24}, the example we provide in this section is simpler. Then we take a step forward and prove that the property of being nearly atomic does not ascend from linearly orderable monoids to their corresponding power monoids.

Motivated by the construction provided in Example~\ref{ex:non-ascent of atomicity}, in this section we will produce an atomic Puiseux monoid whose power monoid is not even nearly atomic, giving a negative answer to the question of whether near atomicity ascends from linearly orderable monoids to their corresponding power monoids.

\begin{lemma} \label{lem: leo 4}
	Let $M$ be a Puiseux monoid, and let $T$ be an element of $\mathcal{P}_\emph{fin}(M)$. Suppose that any divisor $S$ of $T$ in $\mathcal{P}_\emph{fin}(M)$ is such that for any singleton $\{x\}$ dividing $S$ in $\mathcal{P}_\emph{fin}(M)$ there exists a singleton $\{y\} \notin \{\{0\}, S\}$ dividing $S - \{x\}$ in $\mathcal{P}_\emph{fin}(M)$. Then $T$ is not divisible by any atoms.
\end{lemma}

\begin{proof}
    Suppose, by way of contradiction, that $S$ is an atom of $\mathcal{P}_\text{fin}(M)$ that divides $T$. Then $\{0\} \mid_{\mathcal{P}_\text{fin}(M)} S$, so there is some $\{x\} \mid_{\mathcal{P}_\text{fin}(M)} S - \{0\} = S$ where $\{x\}\notin \{\{0\}, S\}$. Since $M$ is reduced, $\{x\}$ and $S - \{x\}$ are not invertible elements, which contradicts that $S$ is an atom.
\end{proof}

Now we present a fact which will become useful to prove the main theorem of this section.

\begin{lemma} \label{lem: leo 5}
	Let $M$ be a Puiseux monoid, and let $p$ be a prime such that there exists a unique atom $a$ in $M$ whose denominator is divisible by $p$. If the map $c_{a,p}$ is constant on $S$ for some $S \in  \mathcal{P}_\emph{fin}(M)$, then $c_{a,p}$ is constant on any divisor $D$ of~$S$ in $\mathcal{P}_\emph{fin}(M)$.
\end{lemma}

\begin{proof}
	Let $S$ be an element of $\mathcal{P}_\text{fin}(M)$ such that the map $c_{p,a}$ is constant on $S$, and let $D$ be a divisor of $S$ in $\mathcal{P}_\text{fin}(M)$. Write $S = C+D$ for some $C \in  \mathcal{P}_\text{fin}(M)$. Then for any $r,s \in D$ and $t \in C$, we see that
	\[
		c_{a,p}(r) + c_{a,p}(t) = c_{a,p}(r+t) = c_{a,p}(s+t) = c_{a,p}(s) + c_{a,p}(t).
	\]
	Hence $c_{a,p}(r) = c_{a,p}(s)$, and so we conclude that $c_{a,p}$ is also constant on $D$.
\end{proof}

We proceed to prove that near atomicity does not ascend from linearly orderable monoids to their corresponding power monoids. This is the primary result of this section.

\begin{theorem}
	There exists an atomic linearly orderable monoid whose power monoid is not even nearly atomic.
    %Puiseux monoid $M$ where $\mathcal{P}_\emph{fin}(M)$ is not nearly atomic.
\end{theorem}

\begin{proof}
    Let $N$, $D$, $Q$, and the terms of the sequence $(P_k)_{k \ge 1}$ be pairwise disjoint infinite subsets of $\pp_{\ge 3}$. Also, let the sequences $(n_i)_{i \ge 1}$, $(d_i)_{i \ge 1}$, and $(q_i)_{i \ge 1}$ be strictly increasing enumerations of the sets $N$, $D$, and $Q$, respectively, and for each $k \in \nn$, let $(p_{k,i})_{i \ge 1}$ be a strictly increasing enumeration of $P_k$. Now let $M$ be the Puiseux monoid generated by the set $\bigcup_{k \in \nn_0} A_k$, where
    \[
        A_0 := \bigg\{ \frac1{2^i q_i} : i \in \nn\bigg\} \quad \text{ and } \quad A_k := \bigg\{ \frac1{p_{k,i}} \bigg( \frac{n_k}{d_k} + \frac1{2^i}\bigg) : i \in \nn \bigg\} 
    \]
    for every $k \in \nn$. Now set $a_{0,i} := \frac1{2^i q_i}$ and $a_{k,i} := \frac1{p_{k,i}} \big( \frac{n_k}{d_k} + \frac1{2^i}\big) \in A_k$ for all $(k,i) \in \nn \times \nn$. After replacing, for each $k \in \nn$, the sequence $(p_{k,i})_{i \ge 1}$ by a suitable subsequence, one can assume that $p_{k,i} >2^i n_k + d_k$ for every $i \in \nn$. Hence the $p_{k,i}$-valuation of $a_{k,i}$ is negative for all $(k,i) \in \nn \times \nn$. Thus, it follows from part~(1) of Lemma~\ref{lem:coefficients of an atom with unique negative p-adic valuation} that $\mathcal{A}(M) = \bigcup_{k \in \nn_0} A_k$ and, therefore, $M$ is atomic. For all $r \in M$ and $i \in \nn$, we set
    \[
    	c_{a_{0,i}}(r) := c_{a_0, q_i}(r) \quad \text{ and } \quad c_{a_{k,i}}(r) := c_{a_{k,i}, p_{k,i}}(r).
    \]
    Now set $b_{k,i} := \frac{n_k}{d_k} + \frac1{2^i}$ for all $(k,i) \in \nn \times \nn$, and observe that $N := \big\langle b_{k,i} : (k,i) \in \nn \times \nn \big\rangle$ is a submonoid of $M$. In addition, it follows from part~(1) of Lemma~\ref{lem:coefficients of an atom with unique negative p-adic valuation} that $N$ is atomic with $\mathcal{A}(N) = \{ b_{k,i} : (k,i) \in \nn \times \nn \}$.

    We will prove that $\ppp := \mathcal{P}_\text{fin}(M)$ is not nearly atomic. Fix $S \in \ppp$ and let us find $T \in \ppp$ such that $S+T$ is not atomic. Because $S$ is a finite set, we can take $j \in \nn$ large enough so that $p_{k,1} > \max \mathsf{d}(S)$ for all $k \ge j$, while $d_j > \max \mathsf{d}(S)$ and $n_j > \max S$. Now set $T := \{t_j, t_{j+1} \}$, where $t_j := b_{j,1}$ and $t_{j+1} := b_{j+1,1}$. 
%    \[
%    	t_j := \frac{n_j}{d_j} + \frac12 \quad \text{ and } \quad t_{j+1} := \frac{n_{j+1}}{d_{j+1}} + \frac12.
%    \]
    %Thus, for any $s \in S$ and $(k,i) \in \nn \times \nn$ with $k \ge j$, we see that $p_{k,i} \nmid \mathsf{d}(s)$, whence $c_{a_{k,i}}(s) = p_{k,i} \zz$ (see part~(2) of Lemma~\ref{lem:coefficients of an atom with unique negative p-adic valuation}).
%    
%    We need the following claim.
%    \smallskip
%    
%    \noindent \textsc{Claim 1.} $a_{k,i} \nmid s$ for any $(k,i) \in \nn_{\ge j} \times \nn$ and $s \in S$.
%    \smallskip
%    
%    \noindent \textsc{Proof of Claim 1.} 
	Fix $(s,t) \in S \times T$ and $c \in \{0,1\}$, and then write 
    \begin{equation} \label{eq:s aux equality}
    	s + c t= r_j + \sum_{k \ge j} \sum_{i \in \nn} c_{k,i} a_{k,i},
    \end{equation}
    where $r_j \in M$ and $\{ c_{k,i} : (k,i) \in \nn_{\ge j} \times \nn \} \subseteq \nn_0$ (with $c_{k,i} = 0$ for almost all $(k,i) \in \nn_{\ge j} \times \nn$) such that $d_k \nmid \mathsf{d}(r_j) \ge 0$ and $p_{k,i} \nmid \mathsf{d}(r_j) \ge 0$ for all $(k,i) \in \nn_{\ge j} \times \nn$. For any $(k,i) \in \nn_{\ge j} \times \nn$, we see that $p_{k,i} \ge p_{k,1} > \max \mathsf{d}(S)$ and so $p_{k,i} \nmid \mathsf{d}(s)$. Thus, for each $(k,i) \in \nn_{\ge j} \times \nn$, we obtain that $p_{k,i} \nmid \mathsf{d}(s + ct)$, and so $p_{k,i} \mid c_{k,i}$. %(see part~(2) of Lemma~\ref{lem:coefficients of an atom with unique negative p-adic valuation}). 
    Therefore %we can rewrite~\eqref{eq:s aux equality} as $s = r_j + \sum_{k \ge j} \sum_{i \in \nn} c'_{k,i} b_{k,i}$, where $c'_{k,i} := c_{k,i}/p_{k,i}$ for all $(k,i) \in \nn_{\ge j} \times \nn$. 
    \begin{equation} \label{eq:s aux equality II}
    	s + ct = r_j + \sum_{k \ge j} \sum_{i \in \nn} \frac{c_{k,i}}{p_{k,i}} \Big( \frac{n_k}{d_k} + \frac1{2^i} \Big) =  r_j + r'_j + \sum_{k \ge j} c_k \frac{n_k}{d_k},
    	%( \frac{n_k}{d_k} + \frac1{2^i} \Big)  +  \sum_{k \ge j} \sum_{i \in \nn} \frac{c_{k,i}}{p_{k,i}} \Big( \frac{n_k}{d_k} + \frac1{2^i} \Big) ,
    \end{equation}
    where $r'_j$ is a nonnegative dyadic rational and $c_k := \sum_{i \in \nn} \frac{c_{k,i}}{p_{k,i}} \in \nn_0$. For each index $k \ge j$, the fact that $d_k \ge d_j > \max \mathsf{d}(S)$ implies that $d_k \nmid \mathsf{d}(s)$. Thus, if $c=0$, then for each index $k \ge j$ we obtain that $\frac{c_k}{d_k} \in \nn_0$ and so the fact that $n_k \ge n_j  > \max S$ (along with~\eqref{eq:s aux equality II}) implies that $c_k = 0$, whence $c_{k,\ell} = 0$ for every $\ell \in \nn$. Thus, any factorization of $s$ in $M$ contains no copy of the atom $a_{k,\ell}$ for any pair $(k,\ell) \in \nn_{\ge j} \times \nn$. %Thus, $a_{k,i} \nmid_M s$ for each pair $(k,i) \in \nn_{\ge j} \times \nn$. 
    Proceeding similarly, we can argue that if $c=1$, then $a_{k,\ell} \nmid s+t$ for each pair $(k,\ell) \in \nn_{\ge j+2} \times \nn$, from which we can deduce that any factorization $z$ of $s+t_j$ (resp., $s+t_{j+1}$) contains no copy of the atom $a_{k,\ell}$ for any pair $(k,\ell) \in \nn_{\ge j+2} \times \nn$ and also that we can pick an index $i \in \nn$ such that $z$ has either $0$ or $p_{j,i}$ (resp., $p_{j+1,i}$) copies of the atom $a_{j,i}$ (resp., $a_{j+1,i}$). %Hence the set of atoms from $A_k$ that appear in any factorization of any element of $T+S$ is either empty or $p_{k,i}$ copies of the atom $ a_{k,i}$ for some index~$i$. Once again, in a similar manner, the same conclusions can be obtained for $A_{k+1}$. 
    Now, for each $s \in S$, define
    \[
        c_{a_{j,i} p_{j,i}}(s) = c_{d_j, a_{j,i} p_{j,i}} \bigg( s - \sum_{m=1}^n c_{a_{j,m}}(s)a_{j,m} \bigg),
    \]
    where $v_{a_{j,m}}(s) = 0$ whenever $m>n$ (a finite $n$ exists as the element $s$ is atomic). Note that this is equal for large enough $i \in \nn$. Before proceeding, we need to establish the following claim.
    \smallskip

    \noindent \textsc{Claim.} Whenever we write $S+T$ as a finite sum of elements of $\ppp$, there is a summand~$S'$ satisfying the following property: for any $s \in S'$, there is another element $t \in S'$ such that one of the following inequalities holds for any sufficiently large index $i \in \nn$:
    \begin{equation} \label{eq:two alternative inequalities}
        c_{a_{j,i} p_{j,i}}(s) > c_{a_{j,i} p_{j,i}}(t) \quad \text{ or } \quad c_{a_{{j+1},i} p_{{j+1},i}}(s) > c_{a_{{j+1},i} p_{{j+1},i}}(t).
    \end{equation}
    %\smallskip

    \noindent \textsc{Proof of Claim.} First, observe that~$T$ satisfies the property in our claim because $c_{a_{j,i} p_{j,i}}(t_k) = c_{a_{{j+1},i} p_{{j+1},i}}(t_{j+1}) = 1$ and $c_{a_{{j+1},i} p_{{j+1},i}}(t_j) = c_{a_{j,i} p_{j,i}}(t_{j+1}) = 0$. Observe that $S$ does not have any factors from $A_j$ or $A_{j+1}$, and so $S+T$ also satisfies the desired property. Now notice that from bounding, these values are actually the exact multiplicities for every factorization. This means that for any $B, C \in \ppp$ such that $B+C \mid_\ppp S+T$, if $a,b \in B$ and $c \in C$, then Lemma~\ref{lem: leo 5} (along with the previous observation) ensures that the term $\sum_{m=1}^n c_{a_{j,m}}(s)a_{j,m}$ does not depend on $a$ and $b$ and, as a consequence,
    \[
        c_{a_{j,i} p_{j,i}}(a+c) - c_{a_{j,i} p_{j,i}}(b+c) = c_{a_{j,i} p_{j,i}}(a) - c_{a_{j,i} p_{j,i}}(b).
    \]
    Write $S+T = S_1 + \dots + S_\ell$ for some elements $S_1, \dots, S_\ell \in \ppp$, and then take an index~$m$ such that $c_{a_{j,m} p_{j,m}}(s) = 0$ for any $s \in S_1 \cup \cdots \cup S_\ell$. If for each summand $S_i$ there exists an element $s_i \in S_i$ for which no corresponding $t$ exists (as in the property of our claim), then $s_1 + \dots + s_\ell$ is an element in $S+T$ with no corresponding $t$, which is not possible because $S+T$ satisfies the desired property. This is because, for each $i \in \ldb 1, \ell \rdb$, the coefficients $v_{a_{j,m} p_{j,m}}(s_i)$ and $c_{a_{{j+1},m} p_{{j+1},m}}(s_i)$ are minimum in $S_i$, and so the corresponding coefficients for $s_1 + \dots + s_\ell$ are minimum in $S_1 + \dots + S_\ell$. Hence the claim is established.
	\smallskip
	
	Now let the summand be $S'$ and take any $s \in S'$. Then the established claim guarantees the existence of $t \in S'$ such that for any sufficiently large $i \in \nn$ one of the inequalities in~\eqref{eq:two alternative inequalities} holds. Assume, without loss of generality, that for any factorization, $c_{a_{j,i} p_{j,i}}(s) > c_{a_{j,i} p_{j,i}}(t)$ holds for any large enough $i \in \nn$. This, together with Lemma~\ref{lem:coefficients of an atom with unique negative p-adic valuation}, implies that the set of atoms from $A_j$ in any factorization of $s$ and $t$ is the empty set or a set of $p_{j,i}$ copies of the atom $a_{j,i}$ for some index~$i$. 
 
    Thus, if $d$ is a common divisor of $\{s,t\}$ in $M$, then the set of atoms from $A_j$ in any factorization of $d$ is the empty set; this means that, for each index $i \in \nn$, there exist $p_{j,i}$ copies of the atom $a_{j,i}$ in any factorization of $s-d$, whence $\frac{n_j}{d_j}+\frac{1}{2^i} \mid_M s$ or $\frac{1}{2^{i+1}} \mid_M s$. Therefore, for any $\{d\} \mid_\ppp S'$, there exists $k \in \nn$ large enough such that $\{\frac{1}{2^k}\} \mid_\ppp S' -\{d\}$. As a consequence, it follows from Lemma~\ref{lem: leo 5} that $S+T$ is not an atomic element of $\ppp$. Hence we can conclude that the power monoid $\ppp$ is not nearly atomic.
\end{proof}

\medskip
%%%%%%%%%%%%%%%%%%%%%%%%%%%%%%%%%%%%%%%%%%%%%%%%%
\subsection{Almost Atomicity and Quasi-Atomicity}

Our primary purpose for the rest of this section is to construct a rank-$2$ linearly orderable monoid that is almost atomic but its power monoid is not even quasi-atomic. With this construction, we provide a negative answer to the ascent of both almost and quasi-atomicity to power monoids.

Consider the Puiseux monoid $Q := \big\langle \frac{1}{p} : p \in \pp \setminus \{2\} \big\rangle$. It is not hard to verify that $Q$ is atomic with $\mathcal{A}(Q) = \big\{\frac1p : p \in \pp \setminus \{2\} \big\}$. Therefore, for each $q \in Q$, we can write
\begin{equation} \label{eq:canonical sum in Q}
	q = \ell(q) + \sum_{p \in \pp \setminus \{2\}} c_p(q) \frac1p,
\end{equation}
where $\ell(q), c_p(q) \in \nn_0$ for every $p \in \pp \setminus \{2\}$. After assuming that we have taken $\ell(q)$ in~\eqref{eq:canonical sum in Q} as large as it can possibly be, we obtain that for each $p \in \pp \setminus \{2\}$ the coefficient $c_p(q)$ is the smallest number of copies of the atom $\frac1p$ that can appear in a factorization of~$q$, which means that $c_p(q) \in \ldb 0, p-1 \rdb$ is the smallest nonnegative integer in $c_{p,1/p}(q)$.  For any $p \in \pp \setminus \{2\}$ and $q \in \text{gp}(Q)$, set $v_{1/p}(q) := \frac{c}p$, where $c$ is the unique element in $\ldb 0, p-1 \rdb$ such that $v_p\big(q - \frac{c}p \big) \ge 0$, and so $v_{1/p}(q) = c_p(q)$ when $q \in Q$. With this in mind, one can extend \eqref{eq:canonical sum in Q} to any $q \in \gp(Q)$ as follows:
\begin{equation}  \label{eq:canonical sum in gp(Q)}
	q = \ell(q) + \sum_{p \in \pp \setminus \{2\}} v_{1/p}(q).
\end{equation}
Now consider the set of positive rationals $N := \text{gp}(Q) \cap (2,3)$ and set $k(q) := 2 - \ell(q) \in \zz$ for each $q \in N$. %For each $q \in \gp(Q)$, set $k(q) = 2 - \big(q - \sum_{p \in \pp \setminus \{2\} } v_{1 / p}(q) \big) \in \ZZ$. 
The most important object for the rest of this section is the following additive submonoid of $\qq^2$:
\begin{equation} \label{eq:main monoid}
	 M := \big\langle A \cup B \cup D \big\rangle,
\end{equation}
where
\[
	A := \Big\{ \Big(\frac{1}{5}, q+\frac{1}{2^{k(q)}} \Big): q \in N \Big\}, \quad B := \Big\{ \Big( \frac{1}{7}, q+\frac{1}{2^{k(q)}} \Big) : q \in N \Big\}, \text{ and } D := \Big\{ \Big(0,\frac{1}{2^n} \Big) : n \in \nn \Big\}.
\]
Since $M$ is a submonoid of the torsion-free abelian group $\qq^2$, it is linearly orderable: indeed, $M$ is a linearly ordered monoid with respect to the lexicographical order on $\qq^2$ with priority on the second coordinate. Observe that $M$ is a reduced linearly orderable monoid. We proceed to argue that $M$ is an almost atomic monoid with $\mathcal{A}(M) = A\cup B$, where every element that is neither the identity element nor an atom is divisible by an element of~$D$.

\begin{lemma} \label{lem:almost atomic of certain monoid}
	The monoid $M$ in~\eqref{eq:main monoid} satisfies the following two conditions.
	\begin{enumerate}
		\item $M$ is almost atomic and $\mathcal{A}(M) = A \cup B$.
		\smallskip
		
		\item $\mathcal{A}(M) + \mathcal{A}(M) \subseteq D+M$.
	\end{enumerate} %is almost atomic, and any non-invertible non-atom is divisible by some element of $D$, that is, 
\end{lemma}

\begin{proof}
	(1) As $M$ is reduced, $\mathcal{A}(M) \subseteq A \cup B \cup D$, from which we obtain that $\mathcal{A}(M) \subseteq A \cup B$. To see that each element of $A$ is an atom of $M$, first notice that no element of $B$ can divide any element of $A$ in $M$ because $\frac17$ does not divide $\frac15$ in the additive monoid $\big\langle \frac15, \frac17 \big\rangle$. Now fix $a_0 \in A$, and note that if $a_0$ is divisible by $a_1 \in A$ in $M$, then $a_0 - a_1$ cannot be divisible in $M$ by any element of $A$, which means that any sum decomposition of $a_0$ in $\langle A \cup D \rangle$ has the form $a_0 = a_1 + (d_1 + \dots + d_n)$ for some $d_1, \dots, d_n \in D$. Thus, if $a_0 = \big(\frac15, q_0 + \frac1{2^{k(q_0)}} \big)$ and $a_1 = \big(\frac15, q_1 + \frac1{2^{k(q_1)}} \big)$ for some $q_0, q_1 \in N$, then the second coordinate $q_0 - q_1$ of $a_0 - a_1$ is a dyadic rational, whence $q_0 - q_1 \in \nn_0$ because $v_2(\gp(Q)) \subseteq \nn_0$. This, along with the fact that $q_0, q_1 \in (2,3)$, ensures that $q_1 = q_0$, which implies that $a_1 = a_0$ and so $d_1 = \dots = d_n = 0$. Hence $A \subseteq \mathcal{A}(M)$ and, similarly, we can conclude that $B \subseteq \mathcal{A}(M)$. Thus, $\mathcal{A}(M) = A \cup B$.
	
	We now prove that $M$ is almost atomic. It suffices to show that any element $(0,\frac{1}{2^n}) \in D$ can be written in $\gp(M)$ as the difference of two atomic elements of $M$. To do so, fix $n \in \nn$. Now take $q,r \in N$ such that $k(q) = n-1$, and let $p$ be an odd prime large enough so that $v_{1/p}(q) = v_{1/p}(r) = 0$ and $q - \frac1p, r+\frac1p \in \gp(Q) \cap (2,3)$. One can readily check that $k\big(q-\frac{1}{p}\big) = k(q)+1 = n$ and $k\big(r + \frac{1}{p}\big) = k(r)$. Since $q-\frac1p$ and $r+ \frac1p$ belong to $N$,
	\[
		a_1 := \pr{ \frac{1}{5}, q + \frac{1}{2^{k(q)}}} + \pr{\frac{1}{5}, r + \frac{1}{2^{k(r)}}} \quad \text{and} \quad a_2 := \pr{\frac{1}{5}, q - \frac{1}{p} + \frac{1}{2^{k(q-\frac{1}{p})}}} + \pr{\frac{1}{5}, r + \frac{1}{p} + \frac{1}{2^{k(r + \frac{1}{p})}}}
	\]
	both belong to $A+A$, and so they are atomic elements of $M$. In addition, from the equalities $k\big( q - \frac1p \big) = n$ and $k\big( r + \frac1p \big) = k(r)$, we can deduce that
	\[
		a_1 - a_2 = \pr{0, \frac1{2^{k(q)}} + \frac1{2^{k(r)}} - \frac1{2^{k\pr{ q - \frac1p }}} - \frac1{2^{k\pr{r + \frac1p }}} } = \pr{ 0, \frac1{2^{n-1}} - \frac1{2^n}} =  \pr{ 0, \frac1{2^n}} .
	\]
	Because $a_1$ and $a_2$ are both atomic elements of the monoid $M$, we can conclude that $M$ is almost atomic.
	\smallskip
	
	(2) We proceed to show that any element of $M$ that can be written as the sum of two atoms must be divisible by an element of~$D$. %We now show any non-atom of $M$ is divisible by an element $T$. It suffices to show that any element which is the sum of two atoms satisfies this, as every non-atom must be divisible by an element of $T$ or the sum of two elements of $A\cup B$. 
	Consider the atoms $a_q = \big( q_0, q + \frac{1}{2^{k(q)}} \big)$ and $a_r = (r_0, r + \frac{1}{2^{k(r)}})$ of~$M$, where $q_0, r_0 \in \big\{\frac15,\frac17\}$ and $q,r \in N$. We have already seen in the previous paragraph that if we take~$p$ to be an odd prime large enough so that $v_{1/p}(q) = v_{1/p}(r) = 0$ and $q - \frac1p, r+\frac1p \in N$, then both equalities $n := k(q) + 1 = k\big(q-\frac{1}{p}\big)$ and $k(r) = k\big(r + \frac{1}{p}\big)$ hold. Thus,
	\[
		a_q + a_r = \pr{q_0, q - \frac{1}{p} + \frac{1}{2^{k(q-\frac{1}{p})}}} + \pr{r_0, r + \frac{1}{p} + \frac{1}{2^{k(r + \frac{1}{p})}}} +  \pr{ 0, \frac1{2^{k(q - \frac1p)}}}.
	\]
	Therefore it follows from the previous equality that $a_q + a_r$ is divisible in $M$ by $\big(0, \frac1{2^n} \big) \in D$ if $n \in \nn$ and by $(0,1) \in D$ if $n \in \zz \setminus \nn$. Hence $\mathcal{A}(M) + \mathcal{A}(M) \subseteq D+M$, as desired.
\end{proof}

%Lemma~\ref{lem:almost atomic of certain monoid} shows that every atom of $\mathcal{P}_\text{fin}(M)$ contains an atom or $0$, as otherwise it is divisible by $\{(0,\frac{1}{2^k})\}$ for some $k$. Call a set that satisfies this property a \textit{semi-atom}. An element of $\mathcal{P}_\text{fin}(M)$ is atomic only if it can be expressed as the sum of atoms, and therefore it must be expressed as a sum of semi-atoms.
% (it turns out this is an if and only if statement)
We are now in a position to prove that although the rank-$2$ monoid in \eqref{eq:main monoid} is almost atomic, its power monoid is not even quasi-atomic, establishing that neither the property of almost atomicity nor that of quasi-atomicity ascend to power monoids on the class of finite-rank torsion-free monoids.

\begin{theorem}
    There exists an almost atomic submonoid of $\mathbb{Q}^2$ such that its power monoid is not quasi-atomic. 
\end{theorem}

\begin{proof}
    Let $M$ be the monoid in \eqref{eq:main monoid}. Set $q := 2 + \frac 13$, and notice that $k(q) = 0$ and, therefore, $\big( \frac15, \frac{10}3 \big) = \big( \frac{1}{5}, q + \frac1{2^{k(q)}} \big) \in A$. In a similar way one can see that $\big( \frac17, \frac{10}3 \big) \in B$. It suffices to show that no nonempty finite subset $S$ of $M$ exists so that $S + \big\{ \big( \frac25, \frac{20}3 \big), \big( \frac37, 10 \big) \big\}$ is atomic. For each $v \in M$, let $\pi(v)$ denote the first coordinate of $v$.
    
    Let us assume, towards a contradiction, that there exist atoms $A_1, \dots, A_n$ of $\Pfin(M)$ such that $\big\{ \big( \frac25, \frac{20}3 \big),\big(\frac37, 10 \big) \big\}$ divides $A_1 + \dots + A_n$ in $\Pfin(M)$. For each $i \in \ldb 1,n \rdb$, let $q_i$ be the minimum of the set $A_i$. Because each $A_i$ is an atom, it follows from Lemma~\ref{lem:almost atomic of certain monoid} that $A_i$ contains either $(0,0)$ or an atom of $M$ (as otherwise it would be divisible by $\big\{ \big( 0, \frac{1}{2^k} \big) \big\}$ in $\Pfin(M)$ for some $k \in \nn$). Therefore $\pi(q_i) \in \big\{ 0, \frac{1}{5}, \frac{1}{7} \big\}$. This means that for any $v \in A_i$, either $\pi(v) = \pi(q_i)$ or $\pi(v) - \pi(q_i) \ge \frac{2}{35}$. Thus, for any $v \in A_1 + \dots + A_n$, either $\pi(v) - \pi\big( \sum_{i=1}^n q_i \big) = 0$ or $\pi(v) - \pi\big( \sum_{i=1}^n q_i \big) \ge \frac{2}{35}$. In other words, $\big{|} \pi(v) - \pi\big( \sum_{i=1}^n q_i \big) \big{|} \neq \frac{1}{35} = \frac37 - \frac 25$. However, $\sum_{i=1}^n q_i \in A_1 + \dots + A_n$, which contradicts the fact that $\{(\frac{2}{5}, \frac{20}{3}), (\frac{3}{7}, 10)\}$ divides $A_1 + \dots + A_n$ in $\Pfin(M)$. Hence the monoid $\Pfin(M)$ is not quasi-atomic.
\end{proof}

Unlike the cases of atomicity and near atomicity, we could not find a rank-$1$ torsion-free almost atomic (resp., quasi-atomic) monoid whose power monoid is not almost atomic (resp., quasi-atomic). Aiming to motivate the search for such rank-$1$ torsion-free monoids, we conclude this section with the following open question.

\begin{question}
    Can we construct a rank-$1$ torsion-free almost atomic (resp., quasi-atomic) monoid whose power monoid is not almost atomic (resp., quasi-atomic)?
\end{question}

\bigskip
%%%%%%%%%%%%%%%%%%%%%%%%%%%%%%%%%%
%%%%%%%%%%%%%%%%%%%%%%%%%%%%%%%%%%
\section{The Furstenberg Property}
\label{sec:Furstenberg and IDF}

In this final section we turn our attention to the Furstenberg and the IDF properties in the setting of finitary power monoids.

\medskip
%%%%%%%%%%%%%%%%%%%%%%%%%%%%%%%%
\subsection{The Furstenberg Property and Weaker Notions}

Similar to atomicity, there are notions of almost and quasi-Furstenberg, which were introduced and studied in~\cite{nLL19} in the setting of integral domains and were then investigated in~\cite{LRZ23} in the setting of Puiseux monoids. Let $M$ be a monoid. We say that~$M$ is \textit{nearly Furstenberg} if there exists $c \in M$ with the following property: for each non-invertible element $b \in M$, there exists $a \in \mathcal{A}(M)$ such that $a \mid_M b+c$ but $a \nmid_M c$. It follows directly from the definitions that every Furstenberg monoid is nearly Furstenberg. We say that~$M$ is \emph{quasi-Furstenberg} (resp., \emph{almost Furstenberg}) if for each non-invertible element $b \in M$, there exist $a \in \mathcal{A}(M)$ and an element (resp., an atomic element) $c \in M$ such that $a \mid_M b+c$ but $a \nmid_M c$. It turns out that the three generalizations of the Furstenberg property we have just defined ascend from linearly orderable monoids to their corresponding power monoids. Before proving this, we need the following lemma.

\begin{lemma} \label{lem:aux Furstenberg}
    Let $M$ be a linearly orderable monoid. For each non-invertible element $S$ of $\ppp_{\emph{fin}}(M)$, there exists either an atom $A$ of $ \ppp_{\emph{fin}}(M)$ with $|A| \ge 2$ such that $A$ divides $S$ in $\ppp_{\emph{fin}}(M)$ or a non-invertible $d \in M$ such that $\{d\}$ divides $S$ in  $\ppp_{\emph{fin}}(M)$.
\end{lemma}

\begin{proof}
    Let $S$ be a nonempty finite subset of $M$ that is invertible in $\Pfin(M)$. Assume that $\{d\}$ does not divide $S$ in $\Pfin(M)$ for any non-invertible $d \in M$. If $S$ is an atom of $\Pfin(M)$, then $S$ cannot be a singleton and so we can take $A := S$. Suppose, otherwise, that $S$ is not an atom of $\Pfin(M)$, and write $S = A+B$ for some non-invertible $A$ and $B$ of $\Pfin(M)$. Among all such sum decompositions, suppose that we have chosen one minimizing $|A|$. Our assumption on $S$ ensures that $|A| \ge 2$ and $|B| \ge 2$, whence it follows from Lemma~\ref{lem:size of the sum} that $|S| > |A|$ and $|S| > |B|$. In this case, $A$ must be an atom of $\Pfin(M)$ as otherwise we could write $A = A' + B'$ for some non-invertible elements $A'$ and $B'$ of $\Pfin(M)$ that are both non-singletons in~$M$, and so the fact that $|B'| \ge 2$ would imply that $|A| > |A'|$, contradicting the minimality of $|A|$.
\end{proof}

We are in a position to establish the ascent of all the Furstenberg-like properties introduced earlier from linearly orderable monoids to their corresponding power monoids.

\begin{theorem}
    Let $M$ be a linearly orderable monoid. Then the following statements hold.
    \begin{enumerate}
         \item If $M$ is a Furstenberg monoid, then $\mathcal{P}_{\emph{fin}}(M)$ is a Furstenberg monoid.
         \smallskip

        \item If $M$ is quasi-Furstenberg, then $\mathcal{P}_{\emph{fin}}(M)$ is quasi-Furstenberg.
        \smallskip

        \item If $M$ is almost Furstenberg, then $\mathcal{P}_{\emph{fin}}(M)$ is almost Furstenberg.
        \smallskip
        
        \item If $M$ is nearly Furstenberg, then $\mathcal{P}_{\emph{fin}}(M)$ is nearly Furstenberg.
    \end{enumerate}
\end{theorem}

\begin{proof}
    Set $\ppp := \mathcal{P}_{\text{fin}}(M)$, and let $S$ be a non-invertible element of $\ppp$.  In light of Lemma~\ref{lem:aux Furstenberg} one of the following two conditions must hold.
    \begin{enumerate}
    	\item[(a)] There exists an atom $A \in \ppp$ with $|A| \ge 2$ such that $A \mid_\ppp S$.
    	\smallskip
    	
    	\item[(b)] There exists a non-invertible element $d \in M$ such that $\{d\} \mid_\ppp S$.
    \end{enumerate}
    \smallskip
    
    (1) Assume that the monoid $M$ is Furstenberg. As $S$ is an arbitrary non-invertible element of $\ppp$, in order to argue that $\ppp$ is a Furstenberg monoid it suffices to argue that $S$ is divisible by an atom in $\ppp$.
	If condition~(a) holds, then we are done because $A$ is an atom of $\ppp$ such that $A \mid_\ppp S$. Therefore suppose that condition~(b) holds. Thus, since $M$ is Furstenberg, we can take $a \in \mathscr{A}(M)$ such that $a \mid_M d$, which implies that $\{a\}$ is an atom of $\ppp$ such that $\{a\} \mid_\ppp S$. In any case, we have found an atom of $\ppp$ that divides $S$. Hence $\ppp$ is a Furstenberg monoid, as desired.
    \smallskip

    (2) This follows similarly to part~(1).
    \smallskip

    (3) This follows similarly to part~(1).
    \smallskip

    (4) Assume now that $M$ is nearly Furstenberg, and let us show that $\ppp$ is also nearly Furstenberg. Take $c \in M$ such that for each non-invertible $b \in M$ the relations $a_b \mid_M b+c$ and $a_b \nmid_M c$ for some $a_b \in \mathcal{A}(M)$. Since $S$ is an arbitrary non-invertible element of~$\ppp$, we only need to argue the existence of an atom~$A \in \ppp$ such that $A \mid_\ppp S + \{c\}$ but $A \nmid_\ppp \{c\}$. If condition~(a) holds, then the inequality $|A| \ge 2$ implies that $A \nmid_\ppp \{c\}$ while the fact that $A \mid_\ppp S$ implies that $A \mid_\ppp S + \{c\}$. Therefore we assume that condition~(b) holds. In this case, we can take $a_d \in \mathcal{A}(M)$ such that $a_d \mid_M d+c$ but $a_d \nmid_M c$. Since $a_d$ is an atom of $M$, we see that $\{a_d\}$ is an atom of $\ppp$. In addition, from the divisibility relations $\{a_d\} \mid_\ppp \{d\} + \{c\}$ and $\{d\} \mid_\ppp S$, we deduce that $\{a_d\} \mid_\ppp S + \{c\}$. Finally, observe that $\{a_d\} \nmid_\ppp \{c\}$ because $a_d \nmid_M c$. Hence we conclude that $\ppp$ is nearly atomic.
\end{proof}

\medskip
%%%%%%%%%%%%%%%%%%%%%%%%%%%%%%
\subsection{The TIDF Property}

In this last section, we consider TIDF-monoids, which are a special type of Furstenberg monoids. Recall that a monoid is a TIDF-monoid provided that the set of divisors of each non-invertible element is nonempty and finite (up to associate). We first prove that the TIDF property ascends to power monoids on the class of positive Archimedean monoids. This is a consequence of the fact that for positive Archimedean monoids the TIDF property implies atomicity. We proceed to prove this last statement.

\begin{prop} \label{prop:positive Archimedean TIDF monoids are atomic}
	Let $M$ be a positive Archimedean monoid. If $M$ is a TIDF-monoid, then $M$ is atomic.
\end{prop}

\begin{proof}
    Assume that $M$ is a TIDF-monoid. Since $M$ is a positive Archimedean monoid, we can suppose that its Grothendieck group $\gp(M)$ is a linearly ordered abelian group under the total order $\preceq$, and also that $M$ is a submonoid of the nonnegative cone of $\gp(M)$ (we can also use H\"older's theorem, and assume that $M \subseteq \rr_{\ge 0}$). Since $M$ is a positive monoid, it must be reduced.
    
	Suppose, by way of contradiction, that the monoid $M$ is not atomic. Then we can take a non-atomic element $q_0 \in M$. Now set $A_1 := \{a \in \mathcal{A}(M) : a \mid_M q_0 \}$. Note that $A_1$ is nonempty and finite because $M$ is a reduced TIDF-monoid and $q_0 \neq 0$. Thus, $A_1$ has a minimum element, and we can set $q_1 := q_0 - \min A_1$. Observe that $q_1$ is not an atomic element because $q_0$ is not an atomic element. For the inductive step of our construction, suppose we have produced, for some $n \in \nn$, a finite descending chain $A_1, \dots, A_n$ of nonempty finite subsets of $\mathcal{A}(M)$ and non-atomic elements $q_1, \dots, q_n \in M$ such that $q_i = q_{i-1} - \min A_i$ for every $i \in \ldb 1,n \rdb$.
%	The fact that $q_0$ is not atomic, along with the equality
%	\[
%		q_0 = q_n + \sum_{i=1}^n \min A_i,
%	\]
%	guarantees that $q_n$ is not atomic. 
	Now set
	\[
		A_{n+1} := \{a \in \mathcal{A}(M) : a \mid_M q_n \}.
	\]
	Because $q_n \mid_M q_{n-1}$, the inclusion $A_n \supseteq A_{n+1}$ holds. In addition, observe that $A_{n+1}$ is a nonempty and finite set because $M$ is a reduced TIDF-monoid and $q_n \neq 0$ (because $q_n$ is non-atomic). Now set $q_{n+1} := q_n - \min A_{n+1}$ and then observe that $q_{n+1}$ is a non-atomic element because $q_n$ is a non-atomic element. After repeating this process indefinitely, we obtain a descending chain $(A_n)_{n \ge 1}$ of nonempty finite subsets of $\mathcal{A}(M)$ and a sequence $(q_n)_{n \ge 0}$ whose terms are non-atomic elements of~$M$ such that $q_n - q_{n+1} =  \min A_{n+1}$ for every $n \in \nn_0$. For each $n \in \nn$, we can now write
	\[
		q_0  = q_n + \sum_{j=0}^{n-1} (q_j - q_{j+1}) = q_n + \sum_{j=0}^{n-1} \min A_{j+1} \succeq n \min A_1.
	\]
	Since $\min A_1 \succ 0$, the fact that $q_0 \succeq n \min A_1$ for every $n \in \nn$ contradicts that~$M$ is a positive Archimedean monoid. 
\end{proof}

We obtain the following corollary on the ascent of both the finite factorization property and the TIDF property.

\begin{cor}
    For a positive Archimedean monoid $M$, the following statements hold.
    \begin{enumerate}
        \item If $M$ is an FFM, then $\Pfin(M)$ is an FFM.
        \smallskip
        
        \item If $M$ is a TIDF-monoid, then $\Pfin(M)$ is a TIDF-monoid.
    \end{enumerate}
\end{cor}

\begin{proof}
    (1) Assume that $M$ is an FFM. Since $M$ is a positive Archimedean monoid, we can take a linearly ordered abelian group $G$ such that $M$ is a submonoid of the nonnegative cone of $G$. By virtue of H\"older's theorem, $G$ is order-isomorphic to a subgroup of the additive group~$\rr$. Therefore $M$ is order-isomorphic to some additive submonoid of $\rr_{\ge 0}$. Since the power monoids of isomorphic monoids are isomorphic, we can assume that $M$ is a submonoid of $\rr_{\ge 0}$. In order to prove now that the power monoid $\Pfin(M)$ is also an FFM, it suffices to follow, \emph{mutatis mutandis}, the argument given in the proof of \cite[Theorem~4.2]{GLRRT24}.
    \smallskip
    
    (2) Assume now that $M$ is a TIDF-monoid. Since $M$ is a positive Archimedean monoid, it follows from Proposition~\ref{prop:positive Archimedean TIDF monoids are atomic} that $M$ is atomic. As $M$ is atomic and every element of $M$ is only divisible by finitely many atoms (up to associate), it follows from \cite[Theorem~2]{fHK92} that $M$ is an FFM. As a consequence, one can deduce from part~(1) that $\Pfin(M)$ is also an FFM. Therefore we conclude that $\Pfin(M)$ is a TIDF-monoid.
\end{proof}

We conclude the paper proving that, in general, the TIDF property does not ascend to power monoids on the more general class of linearly orderable monoids.

\begin{theorem} \label{thm:non-ascent of TIDF in general}
    There exists a linearly orderable monoid $M$ satisfying the following two conditions:
    \begin{enumerate}
    	\item $M$ is a TIDF-monoid and
    	\smallskip
    	
    	\item $\Pfin(M)$ is not an IDF-monoid.
    \end{enumerate}
	%that satisfies the TIDF property whose power monoid does not even satisfy the IDF property.
\end{theorem}

\begin{proof} 
    Let $S := \{a,b,y,z\} \cup \{x_n : n \in \nn\}$ be a set consisting of infinitely many independent variables, and let $\mathcal{G}$ be the free abelian group on~$S$, which can be written as follows:
    \[
        \mathcal{G} = (\ZZ \cdot a) \oplus (\ZZ \cdot b) \oplus (\ZZ \cdot y) \oplus (\ZZ \cdot z) \oplus \bigoplus_{i \in \nn} (\ZZ \cdot x_i).
    \]
	Now set $A := \nn_0 a$ and $B := \nn_0 b$. Then we let $\gp(A)$ and $\gp(B)$ be the Grothendieck groups of the free commutative monoids $A$ and $B$ inside the free abelian group $\mathcal{G}$. For each $n \in \nn$, set
	\[
		X_n := \nn x_n + \gp(A) \quad \text{ and } \quad Y_n := \nn(x_n - y) + \gp(B).
	\]
	Now consider the following submonoid of $\mathcal{G}$:
	\[
		 Z := \nn z + \gp\Big(\Big\langle \bigcup_{n \in \nn} \big( X_n \cup Y_n \big) \Big\rangle \Big).
	\]
	Consider the following submonoid of $\mathcal{G}$:
	\[
		M := \Big\langle A \cup B \cup \Big( \bigcup_{n \in \nn} \big( X_n \cup Y_n\big) \Big) \cup Z \Big\rangle.
	\]
	
    (1) Because $M$ is a submonoid of the free abelian group $\mathcal{G}$, we see that $M$ is a linearly orderable monoid, which means that $M$ is cancellative and torsion-free (by Theorem~\ref{thm:Levi's consequence}). In addition, one can readily verify that $\mathcal{A}(M) = \{a,b\}$ and also that each element of $M$ is divisible by at least one of these two atoms. Therefore $M$ is a TIDF-monoid.
	\smallskip
	
    (2) Now set $\ppp := \Pfin(M)$. As $x_1 - y \in Y_1 \subseteq Z$ and $z - x_1 \in Z$, it follows that $z-y \in Z$ and so $\{z,z-y\} \subset Z \subset M$. We claim that the atoms of $\ppp$ dividing $\{z, z + y\}$ form an infinite set up to associates, from which we deduce that $\ppp$ is not an IDF-monoid. Note that $\{ x_n, x_n - y \} \mid_\ppp \{z, z-y\}$ for every $n \in \nn$. Thus, it suffices to fix $n \in \nn$ and then argue that $\{x_n, x_n-y \}$ is an atom of $\ppp$. To do so, write $\{x_n, x_n-y\} = S + T$ for some $S,T \in \ppp$. It follows from Lemma~\ref{lem:size of the sum} that either $|S| = 1$ or $|T| = 1$. Assume, without loss of generality, that $S$ is a singleton and take $s \in M$ such that $S = \{s\}$. Then we see that $s$ is a common divisor of $x_n$ and $x_n - y$ in $M$. It is clear that $b \nmid_M x_n$ and, therefore, $b \nmid_M s$. On the other hand, the fact that $y \nmid_M x_n$ guarantees that $a \nmid_M x_n - y$, which in turn implies that $a \nmid_M s$. Therefore, neither of the two atoms of $M$ divides $s$. Now the fact that~$M$ is a Furstenberg monoid implies that $s \in \uu(M)$ and so that $S$ is an invertible element of $\ppp$. As a consequence, $\br{x_n, x_n - y}$ is an atom of $\ppp$ for every $n \in \nn$, which means that~$\ppp$ is not even an IDF-monoid.
\end{proof}

\bigskip
%%%%%%%%%%%%%%%%%%%%%%%%%%
%%%%%%%%%%%%%%%%%%%%%%%%%%
\section*{Acknowledgments}

While working on this paper, the authors were part of PRIMES, a year-long math research program hosted by the MIT Math Department. The authors would like to express their gratitude to the directors and organizers of PRIMES for making this research experience possible. During the period of this collaboration, the second author was kindly supported by the NSF under the award DMS-2213323.

\bigskip
%%%%%%%%%%%%%%%%%%%%%%
%%%%%%%%%%%%%%%%%%%%%%
\section*{Conflict of Interest Statement}

On behalf of all authors, the corresponding author states that there is no conflict of interest related to
this paper.

\bigskip
%%%%%%%%%%%%%%
%%%%%%%%%%%%%%


\begin{thebibliography}{20}
	
	\bibitem{AAZ90} D.~D. Anderson, D.~F. Anderson, and M.~Zafrullah, \emph{Factorization in integral domains}, J. Pure Appl. Algebra \textbf{69} (1990) 1--19.

    \bibitem{AM96} D.~D. Anderson and B. Mullins, \emph{Finite Factorization Domains}, Proc. Amer. Math. Soc. \textbf{124} (1996) 389-396.
	
	\bibitem{AG22} D. F. Anderson and F. Gotti, \emph{Bounded and finite factorization domains}. In: Rings, Monoids, and Module Theory (Eds. A. Badawi and J. Coykendall) pp. 7--57. Springer Proceedings in Mathematics \& Statistics, Vol. 382, Singapore, 2022.
		
	\bibitem{BCG21} N. R. Baeth, S. T. Chapman, and F. Gotti, \emph{Bi-atomic classes of positive semirings}, Semigroup Forum \textbf{103} (2021) 1--23.

 	\bibitem{BG23} P. Y. Bienvenu and A. Geroldinger, \emph{On algebraic properties of power monoids of numerical monoids}, Israel J. Math. (to appear). Preprint on arXiv: https://arxiv.org/pdf/2205.00982.pdf

	\bibitem{BC15} J. G. Boynton and J. Coykendall, \emph{On the graph divisibility of an integral domain}, Canad. Math. Bull. \textbf{58} (2015) 449--458.
 
%	\bibitem{CCMS09} P. Cesarz, S.~T. Chapman, S. McAdam, and G.~J. Schaeffer, \emph{Elastic properties of some semirings defined by positive systems}, in Commutative Algebra and Its Applications (Eds. M. Fontana, S.~E. Kabbaj, B. Olberding, and I. Swanson), Proceedings of the Fifth International Fez Conference on Commutative Algebra and its Applications, Walter de Gruyter, Berlin, 2009, pp. 89--101.

%	\bibitem{CGG20} S. T. Chapman, F. Gotti, and M. Gotti, \emph{Factorization invariants of Puiseux monoids generated by geometric sequences}, Comm. Algebra \textbf{48} (2020) 380--396.

    \bibitem{pC17} P. L. Clark, \emph{The Euclidean Criterion for irreducibles}, Amer. Math. Monthly \textbf{124} (2017) 198--216.
 
	\bibitem{pC68} P. M. Cohn, \emph{Bezout rings and their subrings}, Proc. Cambridge Philos. Soc. \textbf{64} (1968) 251--264.
	
%	\bibitem{CG22} J. Correa-Morris and F. Gotti, \emph{On the additive structure of algebraic valuations of polynomial semirings}, J. Pure Appl. Algebra \textbf{226} (2022) 107104.
	
%	\bibitem{CG19} J. Coykendall and F. Gotti, \emph{On the atomicity of monoid algebras}, J. Algebra \textbf{539} (2019) 138--151.

    \bibitem{FT18} Y. Fan and S. Tringali, \emph{Power monoids: a bridge between factorization theory and arithmetic combinatorics}, J. Algebra \textbf{512} (2018) 960-988.

	\bibitem{lF70} L. Fuchs, \emph{Infinite Abelian Groups I}, Academic Press, 1970.

%	\bibitem{aG16} A. Geroldinger, \emph{Sets of lengths}, Amer. Math. Monthly \textbf{123} (2016) 960--988.

%	\bibitem{GGT21} A. Geroldinger, F. Gotti, and S. Tringali, \emph{On strongly primary monoids, with a focus on Puiseux monoids}, J. Algebra \textbf{567} (2021) 310--345.
	
	\bibitem{GH06} A. Geroldinger and F. Halter-Koch, \emph{Non-unique Factorizations: Algebraic, Combinatorial and Analytic Theory}, Pure and Applied Mathematics Vol. 278, Chapman \& Hall/CRC, Boca Raton, 2006.

	\bibitem{rG84} R. Gilmer, \emph{Commutative Semigroup Rings}, The University of Chicago Press, Chicago, 1984.
        
    \bibitem{GLRRT24} V. Gonzalez, E. Li, H. Rabinovitz, P. Rodriguez, and M. Tirador, \emph{On the atomicity of power monoids of puiseux monoids}, International Journal of Algebra and its Applications (to appear). Preprint available on arXiv: https://arxiv.org/pdf/2401.12444
 
%	\bibitem{fG19} F. Gotti, \emph{Increasing positive monoids of ordered fields are FF-monoids}, J. Algebra \textbf{518} (2019) 40--56.

	\bibitem{fG22} F. Gotti, \emph{On semigroup algebras with rational exponents}, Comm. Algebra \textbf{50} (2022) 3--18.

%	\bibitem{GG18} F. Gotti and M. Gotti, \emph{Atomicity and boundedness of monotone Puiseux monoids}, Semigroup Forum \textbf{96} (2018) 536--552.

	%\bibitem{GL22} F. Gotti and B. Li, \emph{Divisibility in rings of integer-valued polynomials}, New York J. Math \textbf{28} (2022) 117--139.

%	\bibitem{GL21} F. Gotti and B. Li, \emph{Semigroup rings and the ascending chain condition on principal ideals}. Preprint submitted, 2021.

	\bibitem{GL23} F. Gotti and B. Li, \emph{Divisibility and a weak ascending chain condition on principal ideals}. Preprint on arXiv: https://arxiv.org/abs/2212.06213
		
%	\bibitem{GP22} F. Gotti and H. Polo, \emph{On the arithmetic of polynomial semidomains}, Forum Mathematicum (to appear). Preprint on arXiv: https://arxiv.org/abs/2203.11478
	
%	\bibitem{GP23} F. Gotti and H. Polo, \emph{On the subatomicity of polynomial semidomains}. In: Algebra and Polynomials: Algebraic, Number Theoretic, and Topological Aspects of Ring Theory (Eds. J. L. Chabert, M. Fontana, S. Frisch, S. Glaz, and K. Johnson). Springer Nature, Switzerland. (to appear) Proceeding of the 2022 Graz Conference in Polynomials and Factorizations. %Preprint on arXiv: %http://arxiv.org/abs/2212.08347

%	\bibitem{GV23} F. Gotti and J. Vulakh, \emph{On the atomic structure of torsion-free monoids}. Preprint on arXiv: http://arxiv.org/abs/2212.08347

	\bibitem{GZ23} F. Gotti and M. Zafrullah, \emph{Integral domains and the IDF property}, J. Algebra \textbf{614} (2023) 564--591.
	
	\bibitem{aG74} A.~Grams, \emph{Atomic rings and the ascending chain condition for principal ideals}, Math. Proc. Cambridge Philos. Soc. \textbf{75} (1974) 321--329.

    \bibitem{GW75} A. Grams and H. Warner, \emph{Irreducible divisors in domains of finite character}, Duke Math. J. \textbf{42} (1975) 271--284.

	\bibitem{fHK92} F. Halter-Koch, \emph{Finiteness theorems for factorizations}, Semigroup Forum \textbf{44} (1992) 112--117.

	\bibitem{nLL19} N. Lebowitz-Lockard, \emph{On domains with properties weaker than atomicity}, Comm. Algebra \textbf{47} (2019) 1862--1868.
 
	\bibitem{fL13} F. W. Levi, \emph{Arithmetische Gesetze im Gebiete diskreter Gruppen}, Rend. Circ. Mat. Palermo \textbf{35} (1913) 225--236.

    \bibitem{LRZ23} A. Lin, H. Rabinovitz, and Q. Zhang, \emph{The Furstenberg property in Puiseux monoids}. Submitted. Preprint on arXiv: https://arxiv.org/abs/2309.12372.
    
 	\bibitem{jP86} J. E. Pin, \emph{Power semigroups and related varieties of finite semigroups}. In: Semigroups and Their Applications (Eds. S.M. Goberstein and P. M. Higgins) pp. 139--152. Proc. Int. Conf. ``Algebraic Theory of Semigroups and Its Applications", California State University at Chico, 1986.

  	\bibitem{mR93} M. Roitman, \emph{Polynomial extensions of atomic domains}, J. Pure Appl. Algebra \textbf{87} (1993) 187--199.
   
	\bibitem{tT86} T. Tamura, \emph{Isomorphism problem of power semigroups of completely 0-simple semigroups}, J. Algebra \textbf{98} (1986) 319--361.

 	\bibitem{sT24} S. Tringali, \emph{On the isomorphism problem for power semigroups}. Preprint on arXiv: https://arxiv.org/abs/2402.11475.
	
	\bibitem{TY23} S. Tringali and W. Yan, \emph{A conjecture by Bienvenu and Geroldinger on power monoids}. Proc. Amer. Math. Soc. (to appear). Preprint on arXiv: https://arxiv.org/abs/2310.17713.
 
\end{thebibliography}
\end{document}